\documentclass[12pt,letterpaper]{article}

\usepackage{amssymb,amsfonts,amscd,amsthm}
\usepackage[all,arc]{xy}
\usepackage{enumerate, bbm}
\usepackage{mathrsfs}
\usepackage{mathtools}
\usepackage{hyperref}
\usepackage[usenames, dvipsnames]{color}
\usepackage{graphicx}
\usepackage{enumitem} 
\usepackage{caption}
\usepackage{subcaption}
\graphicspath{ {TexImages/} }

\newtheorem{thm}{Theorem}[section]
\newtheorem{cor}[thm]{Corollary}
\newtheorem{lem}[thm]{Lemma}
\newtheorem{prop}[thm]{Proposition}

\newtheorem{claim}[thm]{Claim}
\newtheorem{quest}[thm]{Question}

\newtheorem{conj}[thm]{Conjecture}

\hypersetup{
	colorlinks,
	citecolor=blue,
	filecolor=blue,
	linkcolor=blue,
	urlcolor=blue,
	linktocpage
}

\setenumerate[1]{label=\thesection.\arabic*.} 
\setenumerate[2]{label*=\arabic*.} 
\setlist[enumerate]{itemsep=2ex, topsep=2ex} 
\setlist[itemize]{itemsep=2ex, topsep=2ex}

\newcommand{\Z}{\mathbb{Z}}

\newcommand{\del}{\delta}

\renewcommand{\l}{\left}
\renewcommand{\r}{\right}
\newcommand{\half}{\frac{1}{2}}
\newcommand{\sub}{\subseteq}
\newcommand{\ignore}[1]{}

\newcommand{\tr}[1]{\textrm{#1}}
\newcommand{\rec}[1]{\frac{1}{#1}}
\newcommand{\f}[2]{\frac{#1}{#2}}
\newcommand{\floor}[1]{\l\lfloor #1\r\rfloor}
\newcommand{\ceil}[1]{\l\lceil #1\r\rceil}
\newcommand{\Graph}[1]{\includegraphics{#1}}

\newcommand{\quart}{\frac{1}{4}}

\newcommand{\al}{\alpha}
\newcommand{\be}{\beta}
\newcommand{\n}{(n,\al,\be)}
\newcommand{\s}{s^{\al,\be}}
\newcommand{\fibs}{s^{1,1}}
\newcommand{\w}{w^{\al,\be}}
\newcommand{\fibw}{w^{1,1}}
\newcommand{\g}{g^{\al,\be}}
\renewcommand{\a}{a^{\al,\be}}
\renewcommand{\b}{b^{\al,\be}}
\renewcommand{\t}{t^{\al,\be}}
\newcommand{\gam}{\gamma}
\newcommand{\Gam}{\gam_{\al,\be}}
\newcommand{\Om}{\Omega}
\newcommand{\lam}{\lambda}
\newcommand{\Lam}{\lam_{\al,\be}}
\renewcommand{\S}{\mathbf{S}}
\renewcommand{\ss}{\mathbf{s}}

\newcommand{\I}{\mathbf{I}}
\newcommand{\E}{\mathbf{E}}

\newcommand{\sm}{\setminus}
\newcommand{\GAM}{\Gamma}

\title{Slow Recurrences}
\author{
	Sam Spiro\footnote{Dept.\ of Mathematics, UCSD 
		{\tt sspiro@ucsd.edu}.  This material is based upon work supported
		by the National Science Foundation Graduate Research Fellowship under Grant No. DGE-1650112.}
}
\date{\today}
\begin{document}
	\maketitle
	\begin{abstract}
		For positive integers $\alpha$ and $\be$, we define an $(\alpha,\beta)$-walk to be any sequence of positive integers satisfying  $w_{k+2}=\alpha w_{k+1}+\be w_k$.  We say that an $(\alpha,\beta)$-walk is $n$-slow if $w_s=n$ with $s$ as large as possible.  Slow $(1,1)$-walks have been investigated by several authors.  In this paper we consider $(\al,\be)$-walks for arbitrary positive $\alpha,\beta$.  We derive a characterization theorem for these walks, and with this we prove several results concerning the total number of $n$-slow walks for a given $n$.  In addition to this, we study the slowest $n$-slow walk for a given $n$ amongst all possible $\al,\beta$.
	\end{abstract}
	\section{Introduction}
	Let $\al,\be$ be relatively prime positive integers.  Given a pair of positive integers $(a_1,a_2)$, we define its associated $(\al,\be)$-walk $\w_k=\w_k(a_1,a_2)$ to be the sequence with $\w_1=a_1$, $\w_2=a_2$, and $\w_{k+2} = \al \w_{k+1} + \be \w_k$ for $k \geq 1$. 
	
	Given $n$, we are interested in finding $(\al,\be)$-walks which have $n=\w_k$ for some $k$.  Trivial examples of this include $\w_k(x,n)$ for any $x$.  In order to make the problem more interesting, we will demand that our walks generate $n$ as slowly as possible.  To this end, define $\s(n;a_1,a_2)$ to be the integer $s$ such that $\w_s(a_1,a_2) = n$, with this value being $-\infty$ if no such $s$ exists. 
	Let $\s(n) = \displaystyle\max_{a_1,a_2 \geq 1} \s(n;a_1,a_2)$. We will say that $(a_1,a_2)$ is $\n$-good if $a_1, a_2 \geq 1$ and if $\s(n) = \s(n;a_1,a_2)$.
	If $(a_1,a_2)$ is an $\n$-good pair, then we will say that its associated sequence $\w_k(a_1,a_2)$ is an $n$-slow $(\al,\be)$-walk.
	
	For example, it is easy to verify that $\fibs(6)=4$ and that the only 6-slow  (1,1)-walks are $\fibw_k(2,2)$ and $\fibw_k(4,1)$.  As another example, observe that $\s(1)=2$ for all $\al,\be \ge 1$ since $\w_3(a_1,a_2)=\al a_1+\be a_2\ge 2$ when $a_1,a_2\ge 1$.  Thus for all $\al,\be$ we have that $(x,1)$ is a $(1,\al,\be)$-good pair for all $x$.  Similarly, if $\s(n)=2$, then $(x,n)$ will be an $\n$-good pair for all $x$.  Because of this, the only $n$ which have ``interesting'' $n$-slow walks are those with $\s(n)>2$.
	
	The case $\al=\be=1$ has been considered several times \cite{S1, S2, S3}, and a number of results for this case have been produced independently by several authors \cite{CGS, EB, JK}.  In particular, one can prove a certain ``characterization theorem'' that completely classifies $(n,1,1)$-good pairs given any choice of $n$.
	
	The main goal of this paper is to prove a generalized characterization theorem for $(\al,\be)$-walks.  When $\be=1$ this result and proof is quite similar to the $\al=\be=1$ case.  However, for $\be>1$ the situation becomes somewhat more complicated.  To state our result, we first define a generalization of the Fibonacci numbers.  Let $\g_k$ denote the sequence with
	\[
		\g_1=1,\ \g_2=\al,\]\[ \g_{k+2}=\al\g_{k+1}+\be\g_k\tr{ for }k\ge 1.
	\]
	Also define \[\Gam=\half(\al+\sqrt{\al^2+4\be}),\ \Lam=\half (\al-\sqrt{\al^2+4\be}).\] 
	Note that if $f_k$ denotes the $k$th Fibonacci number and $\phi$ denotes the golden ratio, then we have $g_k^{1,1}=f_k$ and $\gam_{1,1}=\phi$.

	\begin{thm}\label{T-GMain}
		Let $\al,\be$ be relatively prime positive integers.  Let $n$ be such that $s=\s(n)>2$.  Then there exists unique integers $a=\a(n),\ b=\b(n)$, and $t=\t(n)$ satisfying the following.
		\begin{itemize}
			\item $n=a\g_t+\be b\g_{t-1}$.
			\item $t\ge 2,\ a\le (\be-1)\g_{t+1}+\al b,$ and $b\le \g_{t}$.
			\item $a-\al b- \ell \g_{t+1}$ is not a positive multiple of $\be$ for any $\ell \ge 0$.
		\end{itemize}
		Further, we have the following.
		\begin{itemize}
			\item The pair $(b,a)$ is $\n$-good and $s=t+1$.
			\item We have $|\w_{s+1}(b,a)-\Gam n|=|\Lam^t(\Gam b-a)|\le 2\be^{t+1}$.
			\item A pair $(b',a')$ of positive integers is $(n,\al,\be)$-good if and only if $(b',a')=(b+kg_t^{\al,\be},a-k\be g_{t-1}^{\al,\be})$ for some $k\ge 0$.   
			\item With $a',b',k$ as above, we have $\w_{s+1}(b',a')-\w_{s+1}(b,a)=k(-\be)^{t}$.
		\end{itemize}
	\end{thm}
	In Lemma~\ref{L-Chicken} we show that $\s(n)>2$ provided $n>\al\be$, so for any fixed $\al,\be$ this theorem applies to all but finitely many $n$.
	
	For the rest of the paper we omit writing $\al,\be$ whenever these are clear from context.  For example, we write $g_k$ instead of $\g_k$, say that a pair is $n$-good instead of $(n,\al,\be)$-good, and so on.  We also will assume $\al,\be$ are positive relatively prime integers unless stated otherwise.  As usual, let $\floor{x}$ denote the floor of $x$ and let $\ceil{x}$ denote the ceiling of $x$.

	One can prove a number of results concerning slow walks by utilizing Theorem~\ref{T-GMain}. For example, let $p(n)=p^{\al,\be}(n)$ denote the number of $n$-good pairs.
	
	\begin{thm}\label{T-Pairs}
		Let $\al,\be$ be relatively prime positive integers.
		\begin{itemize}
			\item If $s(n)>2$, then $p(n)\le \al^2+2\be-1$.   Moreover, there exists an $n$ achieving this bound.
			\item There exist infinitely many $n$ with \[p(n)=\ceil{\gam^2}-1=\al^2+\be+\ceil{\al\be\gam^{-1}}-1,\] and only finitely many $n$ with \[p(n)\ge\ceil{\gam^2}=\al^2+\be+\ceil{\al\be\gam^{-1}}.\]
			\item There exists infinitely many $n$ with $p(n)=\al^2+2\be-1$ if and only if $\al\ge \be$.
		\end{itemize}
	\end{thm}
	One can prove a density result for the number of $n$ with $p(n)>p$.  To this end, for $p$ an integer, let $S_p=\{m:p(m)>p\}$.  We emphasize that $S_p$ is the set consisting of $n$ with \textit{strictly} more than $p$ good pairs.  For $c$ real and $r$ an integer, define $n_{c,r}:=\floor{\f{c\be}{(\gam-\lam)^2} \gam^{2r+1}}$.
	
	\begin{thm}\label{T-Den}
		Let $\al,\be$ denote relatively prime positive integers.  Given an integer $p$, let $d$ denote the smallest integer such that $ \del:=\be\gam^{-1}p-\gam d\le \al$.	If $\be\le p\le \ceil{\gam^2}-2$ and $1\le c\le (p-\be+1)\gam/\al$, then $n_{c,r}^{-1}|S_p\cap [n_{c,r}]|=$
		{\footnotesize \begin{align*}c^{-1}\l(\f{(2\be-2d-1)\gam(\al-2\del+\al^{-1}\del^2)}{2\be^2(\gam^2-1)}+\f{\gam^2}{\gam^2-1}\sum_{q=d+1}^{\be-1}\f{\be-q}{\be^2}\r)+O(\gam^{-r}+(\be\gam^{-2})^r).\end{align*}}
	\end{thm}

	A few comments are in order.  The first is that for $p\ge \ceil{\gam^2}-1$ the density of $S_p$ is 0 by Theorem~\ref{T-Pairs}, so it makes sense to only consider $p\le \ceil{\gam^2}-2$.  Because $\be\gam^{-2}<1$ for all $\al$, this error term is $o(1)$, though it is the case that $\be\gam^{-2}$ can be made arbitrarily close to 1.  We note that this result implies in particular that $\lim_{n\to \infty} n^{-1}|S_p\cap [n]|$ does not exist, and instead this exact value oscillates with $n$.  We have included concrete examples and data illustrating these density results in the Appendix.
	
	The statement of Theorem~\ref{T-Den} becomes significantly cleaner if $d=\be-1$, which occurs when $p$ is sufficiently large.  Specifically, we will show the following.
	
	\begin{cor}\label{C-pLarge}
		Let $\al,\be$ be relatively prime integers.  For $p$ an integer, let $\del=\be\gam^{-1}p-\gam (\be-1)$. If $ \max\{\be,\floor{(1-\be^{-1})\gam^2}\}\le p\le\ceil{\gam^2}-2$ and $1\le c\le (p-\be+1)\gam/\al$, then
		\begin{align*}
		n_{c,r}^{-1}|S_p\cap [n_{c,r}]|=\f{\gam(\al-2\del+\al^{-1}\del^2)}{2c\be^2(\gam^2-1)}+O(\gam^{-r}+(\be\gam^{-2})^r).
		\end{align*}
	\end{cor}

	When $\be=1$, the statement of Corollary~\ref{C-pLarge} becomes even simpler.
	\begin{cor}\label{C-be=1}
		Let $\al$ and $1\le p\le \al^2$ be integers and $1\le c\le p\gam/\al$ real. Then
		\begin{align*}
		n_{c,r}^{-1}|S_p\cap [n_{c,r}]|=\half c^{-1}(1-(\al\gam)^{-1}p)^2+O(\gam^{-r}).
		\end{align*}
		In particular, for $p=\al^2$ and $1\le c\le \al\gam$ we have
		\begin{align*}
		n_{c,r}^{-1}|S_p\cap [n_{c,r}]|=\half c^{-1} \gam^{-4}+O(\gam^{-r}).
		\end{align*}
	\end{cor}

	Given $n$, one might ask how slow the slowest $n$-slow walk is.  To this end, we define \[\ss(n)=\max_{\substack{\al,\be\ge 1,\\ \gcd(\al,\be)=1}} \s(n),\]
	as well as the set of  $\al,\be$ pairs achieving this value, \[ \S(n)=\{(\al,\be):\s(n)=\ss(n)\}.\]
	
	A priori, any pair $(\al,\be)$ could be an element of $\S(n)$ for some $n$.  However, it turns out that only a finite number of pairs have this property.
	\begin{thm}\label{T-Slowest}
		Let $R=\{(1,1),(2,1),(1,2),(1,3),(1,4)\}$.
		\begin{itemize}
			\item For all $n>1$, we have $\S(n)\sub R$.  
			\item For all $(\al,\be)\in R$, there exists an $n$ with $(\al,\be)\in \S(n)$.
			\item The set of $n$ with $\S(n)=\{(1,1)\}$ has density 1.
		\end{itemize}  
	\end{thm}
	
	The rest of the paper is organized as follows.  In Section~\ref{S-alMain} we prove Theorem~\ref{T-GMain}.  In Section~\ref{S-App} we apply this result to prove Theorem~\ref{T-Pairs}, along with several other results about slow walks. In Section~\ref{S-Den} we develop a general theory for proving density results for $n$-slow walks, and we use this theory to prove Theorem~\ref{T-Den}.  In Section~\ref{S-SS} we prove results concerning slowest $n$-slow walks.  We end with some open problems in Section~\ref{S-Con}.
	
	Throughout this paper we use various identities and inequalities involving $\gam$ and $\lam$ to simplify our expressions.  Essentially all of these identities follow from the fact that these are the two solutions to the equation $x^2=\al x+\be$, and all of the inequalities are immediate from the definitions.  In particular, we will make frequent use of the following: 
	\begin{align*}
	\gam^2&=\al \gam+\be,\\ \gam+\lam&=\al,\\ 
	\gam\lam&=-\be,\\ 
	\lam&<0,\\
	\gam&>\al,\ \be^{1/2},\ |\lam|.\\
	\end{align*}
\section{The Characterization Theorem}\label{S-alMain}
Recall that we assume that $\al$ and $\be$ are relatively prime positive integers.  We adopt the convention that $g_0=0$ and $g_{-1}=\be^{-1}$.  Note that with this we have $g_{k+2}=\al g_{k+1}+\be g_k$ for all $k\ge -1$.

The sequence $g_k$ is very similar to the Fibonacci numbers $f_k$.  In particular, we have the following.

\begin{lem}\label{L-FibLike}
	~
	\begin{itemize}
		\item[(a)] For $k\ge 1$, $\gcd(g_k,\be)=1$.
		\item[(b)] For $k\ge 1$, $\gcd(g_{k+1},\be g_{k})=1$.
		\item[(c)] For $k\ge -1$,  $g_{k+1}^2-g_{k}g_{k+2}=(-\be)^k$.
		\item[(d)] For $k\ge -1$, we have $g_k=\f{\gam^k-\lam^k}{\gam-\lam}$.
	\end{itemize}
\end{lem}
\begin{proof}
	Observe that (a) and (b) hold for $k=1$, and that (c) holds for $k=-1,0$.  For (a), inductively assume that we have proven the statement up to $k>1$.  We then have
	\[
	\gcd(g_k,\be)=\gcd(\al g_{k-1}+\be g_{k-2},\be)=\gcd(\al g_{k-1},\be)=1,
	\]
	where we used that both $\al$ and $g_{k-1}$ are relatively prime to $\be$.
	
	For (b), inductively assume we have proven the statement up to $k>1$.  We then have
	\[
	\gcd(g_{k+1},g_{k})=\gcd(\al g_k+\be g_{k-1},g_k)=\gcd(\be g_{k-1},g_k)=1.
	\]
	This shows that $g_{k+1}$ is relatively prime to $g_k$.  We know that $g_{k+1}$ is also relatively prime to $\be$ by (a), so we conclude the result.
	
	For (c), inductively assume we have proven the statement up to $k>0$.  We then have
	\begin{align*}
	g_{k+1}^2-g_{k}g_{k+2}&=(\al g_k+\be g_{k-1})g_{k+1}-(\al g_{k+1}+\be g_k)g_{k}\\ 
	&=\be g_{k-1}g_{k+1}-\be g_k^2=-\be(-\be)^{k-1}=(-\be)^k.
	\end{align*}
	
	The formula in (d) is well known, see \cite{W2}, for example.  The statement can also be proven inductively, or by plugging the formula into the recurrence relation.
\end{proof}

We can use $g_k$ to write formulas for $w_k(b,a)$ and $w_{k+1}(b,a)$. 

\begin{lem}\label{L-Form}
	The following formulas hold for $a,b,k\ge 1$.
	\begin{itemize}
		\item[(a)] \[w_k(b,a)=ag_{k-1}+\be bg_{k-2}.\]
		\item[(b)] \[w_{k+1}(b,a)=\gam w_k(b,a)+\lam^{k-1}(a-\gam b).\]
	\end{itemize}
\end{lem}
\begin{proof}
	The statement for (a) is true for $k=1,2$, and an easy induction argument gives the full result.  For (b), we use (a) and the closed form for $g_k$ provided by Lemma~\ref{L-FibLike}(d) to conclude that
	\begin{align*}
	(\gam-\lam)w_{k+1}(b,a)&=a(\gam^k-\lam^k)+\be b(\gam^{k-1}-\lam^{k-1})\\ &=(a(\gam^k-\gam \lam^{k-1})+\be b(\gam^{k-1}-\gam \lam^{k-2})\\&+a(\gam \lam^{k-1}-\lam^k)+\be b(\gam \lam^{k-2}-\lam^{k-1})\\ &=\gam (\gam-\lam)w_{k}(b,a)+\lam^{k-1}(\gam-\lam)(a+\be\lam^{-1}b).
	\end{align*}
	Dividing both sides by $\gam-\lam$ and using $\lam^{-1}=-\be^{-1} \gam$ gives the final result.
\end{proof}

\begin{cor}\label{C-Form}
	Let $n$ be a positive integer with $s(n)=s$.  Then there exists no $a',b'\ge 1$ with $n=a'g_s+\be b'g_{s-1}$.  Moreover, $(b,a)$ is an $n$-good pair if and only if $n=ag_{s-1}+\be bg_{s-2}$ and $a,b\ge 1$.
\end{cor}
\begin{proof}
	If there existed such $a',b'$, then this would imply that $w_{s+1}(b',a')=n$ by Lemma~\ref{L-Form}(a), a contradiction to how $s$ was defined.  If $(b,a)$ is an $n$-good pair, then again by this lemma we have $n=w_s(b,a)=ag_{s-1}+\be bg_{s-2}$.  Conversely, if $n= ag_{s-1}+\be bg_{s-2}$, then $w_k(b,a)$ satisfies $w_s(b,a)=n$, so this pair is $n$-good.
\end{proof}

We can now derive some structural results for $n$-good pairs.

\begin{lem}\label{L-Facts}
	Let $(b,a)$ be an $n$-good pair with $s=s(n)>2$.
	\begin{itemize}
		\item[(a)] $a-\al b-\ell g_s$ is not a positive multiple of $\be$ for any $\ell\ge 0$.
		\item[(b)] The pair $(b',a')$ is $n$-good if and only if there exists some $k\in \Z$ such that $a'=a+k\be g_{s-2}\ge 1$ and $b'=b-kg_{s-1}\ge 1$.
		\item[(c)] With $a'$, $b'$, and $k$ as above, we have $w_{s+1}(b',a')-w_{s+1}(b,a)=k(-\be)^{s-1}$.
	\end{itemize}
\end{lem}
\begin{proof}
	For (a), assume that this were not the case for some $\ell\ge 0$.  Set $a'=b+\ell g_{s-1}$ and $b'=\be^{-1}(a-\al b-\ell g_s)$, noting that by assumption these are positive integers.  With this we have that $n$ is equal to
	\begin{align*}
	ag_{s-1}+\be bg_{s-2}&=(a-\al b+\al b+\ell \al g_{s-1}+\ell \be g_{s-2}-\ell g_s)g_{s-1}+\be bg_{s-2}\\&=(a-\al b+\al a'-\ell g_s)g_{s-1}+\be(b+\ell g_{s-1})g_{s-2}\\&=(\al a'+\be b')g_{s-1}+\be a' g_{s-2}\\&=a'g_s+\be b' g_{s-1}.
	\end{align*}
	This contradicts Corollary~\ref{C-Form} and the assumption $s(n)=s$, so we conclude (a).
	
	For (b), note that by Corollary~\ref{C-Form} we have that $(b',a')$ is $n$-good if and only if this is a positive solution to the Diophantine equation $n=a'g_{s-1}+\be b'g_{s-2}$.  The result follows from Lemma~\ref{L-FibLike}(b) since $s>2$.
	
	For (c), we have by Lemma~\ref{L-Form}(a) and Lemma~\ref{L-FibLike}(c) that
	\begin{align*}
	w_{s+1}(b',a')-w_{s+1}(b,a)&=(a+k\be g_{s-2})g_s+(b-kg_{s-1})\be g_{s-1}-ag_s-bg_{s-1}\\&=-k\be (g_{s-1}^2-g_sg_{s-2})=-k\be (-\be)^{s-2}\\&=k(-\be)^{s-1}.
	\end{align*}
\end{proof}

To prove the next lemma, we make use of the following special case of the Frobenius coin problem~\cite{Syl}.
\begin{lem}\label{L-Coin}
	If $x,y,n$ are positive integers with $\gcd(x,y)=1$ and $n>xy$, then there exist integers $a,b\ge 1$ such that $n=ax+by$.
\end{lem}

\begin{lem}\label{L-abBound}
	If $(b,a)$ is $n$-good with $s(n)=s>2$, then $a\le (\be-1)g_s+\al b$.
\end{lem}
\begin{proof}
	If $a-\al b+g_s>\be g_s$, then because $\gcd(\be,g_s)=1$ by Lemma~\ref{L-FibLike}(b), there exist positive $r,\ell$ such that $a-\al b+g_s=r\be+\ell g_s$ by Lemma~\ref{L-Coin}. This implies that $a-\al b-(\ell-1) g_s$ is a positive multiple of $\be$. This contradicts Lemma~\ref{L-Facts}(a), so we conclude that $a-\al b\le (\be-1)g_s$.
\end{proof}

From this proof, we see that the condition $a\le (\be-1)g_s+\al b$ in Theorem~\ref{T-GMain} is already implied by the condition stating that $a-\al b-\ell g_{t+1}$ is not a positive multiple of $\be$ and hence could be omitted.  However, we feel it is useful to state this condition explicitly.

We are now ready to prove our characterization theorem.

\begin{proof}[Proof of Theorem~\ref{T-GMain}]
	Let $(b',a')$ be an $n$-good pair, and let $k$ be the non-negative integer such that $1\le b'-kg_{s-1}\le g_{s-1}$. By Lemma~\ref{L-Facts}(b) we have that $(b,a):=(b'-kg_{s-1},a'+k\be g_{s-2})$ is an $n$-good pair, and hence $n=ag_{s-1}+\be bg_{s-2}$ by Corollary~\ref{C-Form}. By construction we have $b\le g_{s-1}$, and because the pair $(b,a)$ is $n$-good, we also have $a\le (\be-1)g_s+\al b$ and that $a-\al b- \ell g_{s}$ is not a positive multiple of $\be$ for any $\ell \ge 0$ by Lemmas~\ref{L-abBound} and \ref{L-Facts}.  By taking $a^{\al,\be}(n)=a,\ b^{\al,\be}(n)=b$, and $t^{\al,\be}(n)=s-1$, we conclude that such integers exist.
	
	We next show that these integers are unique.  Assume that $n=ag_t+\be bg_{t-1}$ with $a,b,t$ as in the hypothesis of the theorem.  This implies that $s\ge t+1$ by Lemma~\ref{L-Form}(a). Assume that $s> t+1$. This implies by Lemma~\ref{L-Form}(a) that there exist $a',b'\ge 1$ such that \[ag_t+\be bg_{t-1}=n=a'g_{t+1}+\be b'g_t=(\al a'+\be b')g_t+\be a'g_{t-1}.\]  Because $t\ge 2$ by assumption, $g_{t}$ and $\be g_{t-1}$ are relatively prime.  Thus having two solutions to the Diophantine equation $n=x g_t+y\be g_{t-1}$ implies that there exists a $k$ such that
	\begin{align*}
	a'&=b+kg_t\\ 
	\be b'&=a-\al a'-k\be g_{t-1}=a-\al b-k\al g_t-k\be g_{t-1}\\&=a-\al b-kg_{t+1}.
	\end{align*}
	Because $b\le g_{t}$ and $a'\ge 1$, we must have $k\ge 0$ in order for the first equation to hold.  We assumed for $\ell \ge 0$ that $a-\al b-\ell g_{t+1}$ is not a positive multiple of $\be$, and since $b'\ge 1$ this implies that we can not have $k\ge 0$.  We conclude that no such $k$ exists, and hence we must have $s(n)=t+1$, from which it follows that $(b,a)$ is an $n$-good pair.  This is the unique pair with $1\le b\le g_t$ by Lemma~\ref{L-Facts}(b), so we conclude that $a,b$ are also unique.
	
	For the final results, note that in this proof we have already shown that $s=t+1$, and the results concerning $(b',a')$ follow from Lemma~\ref{L-Facts}.  Thus it only remains to bound $|w_{s+1}(b,a)-\gam n|$.  Because $w_s(b,a)=n$, it follows from Lemma~\ref{L-Form}(b)  that this is equal to $|\lam^{s-1}(a-\gam b)|=|\lam^t(a-\gam b)|$.  In order to bound this quantity, note that \begin{equation*}\label{E-alMain}a-\gam b\le (\be-1)g_{t+1}+(\al-\gam)b\le \f{\be-1}{\gam-\lam}(\gam^{t+1}-\lam^{t+1}),\end{equation*}
	where we used that $\al<\gam$ and the closed form of $g_{t+1}$.  On the other hand, we have
	\[
		a-\gam b> - \gam g_t=\f{-\gam}{\gam-\lam}(\gam^t-\lam^t).
	\]
	By using the rough bounds $\gam^k-\lam^k \le 2\gam^k$ and $1,\be-1\le \be$, we in total find that
	\[
		|a-\gam b|\le \f{2\be }{\gam-\lam} \gam^{t+1}.
	\]
	Multiplying this by $|\lam^t|$ and using $\lam \gam=-\be$ gives
	\[
		|\lam^t(a-\gam b)|\le \f{2\gam }{\gam-\lam}\be^{t+1},
	\]
	Because $-\lam >0$, we conclude the result.
\end{proof}
We can prove a somewhat nicer result when $\be=1$. 
\begin{cor}\label{C-alMain}
	Let $n$ be such that $s=s^{\al,1}(n)>2$ and let $a,b,t$ be the integers of Theorem~\ref{T-GMain}.
	\begin{itemize}
		\item The values $a,b,t$ are the unique integers satisfying $n=ag_t^{\al,1}+ bg_{t-1}^{\al,1},\ t\ge 2$, and $1\le a\le \al b\le \al g_t^{\al,1}$.  
		\item Every $(n,\al,1)$-good pair $(b',a')$ is of the form $(b+kg_t^{\al,1},a-k g_{t-1}^{\al,1})$ for some $k\ge 0$.  
		\item With $a',b',k$ as above, we have $w_{s+1}^{\al,1}(b',a')=\floor{\gam_{\al,1} n}-k$ if and only if $t$ is even and $w_{s+1}^{\al,1}(b',a')=\ceil{\gam_{\al,1} n}+k$ if and only if $t$ is odd.
	\end{itemize}
\end{cor}
\begin{proof}
	Because $\be=1$, we have by Theorem~\ref{T-GMain} that $a\le \al b$, and from this it is immediate that $a-\al b-\ell g_{t+1}$ is not positive for any $\ell\ge 0$, so we can drop this condition.  It remains to prove the results concerning $w_{s+1}$.
	
	Assume that $t$ is even.  We know from Theorem~\ref{T-GMain} that $s=t-1$ and that \[w_{s+1}(b',a')-w_{s+1}(b,a)=k(-1)^s=-k,\] so it suffices to show that $w_{s+1}(b,a)=\floor{\gam n}$.  Observe that this is equivalent to having $\gam n-w_{s+1}(b,a)$ being non-negative and less than one. Since $n=w_s(b,a)$, by Lemma~\ref{L-Form}(b) this is equivalent to having
	\begin{equation}\label{E-Floor}
	0\le \lam^{t}(\gam b-a)<1.
	\end{equation}
	Note that when $\be=1$ we have $-1<\lam<0$ for all $\al\ge 1$.  Because $t$ is even, the only way \eqref{E-Floor} could be negative is if $\gam b-a<0$, but this is impossible since $a\le \al b<\gam b$. For the upper bound, using $\gam\lam=-1$ and that $t$ is even, we have
	\begin{align*}
	\lam^t(\gam b-a)&\le -\lam^{t-1}g_t-\lam^t=\f{-(\lam \gam)^t\lam^{-1}-\lam^{2t-1}}{\gam-\lam}-\lam^t
	\\&=\f{(-1)^t}{1+\lam^2}-\lam^t(1+\f{\lam^{t-1}}{\gam-\lam})\\&=\f{1}{1+\lam^2}-\lam^t(1+\f{\lam^{t-1}}{\gam-\lam})\\&<1-\lam^t(1+\f{\lam}{\gam-\lam})=1-\lam^t(\f{\gam}{\gam-\lam})<1,
	\end{align*}
	where we obtained the first strict inequality by using $|\lam|<1$ and that $t\ge 2$.  This proves the result for $t$ even, and the proof for $t$ odd is essentially the same.  We omit the details.
\end{proof}

\section{Applications}\label{S-App}
An immediate corollary of the second to last point of Theorem~\ref{T-GMain} is the following.
\begin{cor}\label{C-Pairs}
	Let $n$ be such that $s(n)>2$ and let $a=a(n),\ b=b(n)$, and $t=t(n)$ be as in Theorem~\ref{T-GMain}.  Then for any integer $p$, there are more than $p$ distinct $n$-good pairs if and only if $a>\be p g_{t-1}$.
\end{cor}
\begin{proof}
	If $a>\be p g_{t-1}$, then the pairs $(b+k g_t,a-k\be g_{t-1})$ with $0\le k\le p$ are all $n$-good.  Otherwise strictly fewer $k$ work.
\end{proof}

In order to prove Theorem~\ref{T-Pairs}, it will be convenient to make use of the following.
\begin{lem}\label{L-Ext}
	For $t\ge 2$, let $n_t=\be g_tg_{t+1},\ a_t=(\be-1)g_{t+1}+\al g_t$, and $b_t=g_t$. In the notation of Theorem~\ref{T-GMain}, we have $a(n_t)=a_t$, $b(n_t)=b_t$, and $t(n_t)=t$.
\end{lem}
\begin{proof}
	One can verify that $n=a_t g_{t}+\be b_tg_{t-1}$ and that the relevant bounds of Theorem~\ref{T-GMain} hold for $t,a_t,b_t$. For any $\ell$, $a_t-\al b_t-\ell g_{t+1}=(\be-1-\ell)g_{t+1}$.  If $\ell\ge \be-1$ this quantity is not positive.  If $0\le \ell<\be-1$ we have $\gcd(\be,\be-1-\ell)<\be$, and because $\gcd(\be,g_{t+1})=1$ by Lemma~\ref{L-FibLike}, we conclude that $a_t-\al b_t-\ell g_{t+1}$ is never a positive multiple of $\be$ for $\ell\ge0$.  Thus $a_t,b_t$ must be equal to $a(n_t),b(n_t)$ by the uniqueness of these integers.
\end{proof}
We are now ready to prove our bounds on the number of $n$-good pairs.
\begin{proof}[Proof of Theorem~\ref{T-Pairs}]
	Given $n$, let $a,b,t$ be as in Theorem~\ref{T-GMain}.  By Corollary~\ref{C-Pairs}, $p(n)$ is the largest integer such that \begin{equation}\label{E-p}a>(p(n)-1) \be g_{t-1}.\end{equation}  Because $b\le g_t$ and $a\le (\be-1)g_{t+1}+\al b$, we have $a\le a_t$ as defined in Lemma~\ref{L-Ext}.  It follows that $p(n)\le p(n_t)$, so it will be enough to bound $p(n_t)$.  Observe that \begin{align}a_t&= (\be-1)g_{t+1}+\al g_t=(\be-1)(\al g_t+\be g_{t-1})+\al g_t\nonumber\\ &=\al\be g_t+(\be-1)\be g_{t-1}=(\al^2+\be-1 )\be g_{t-1}+\al\be^2 g_{t-2}.\nonumber\end{align}
	
	Let $p_t=p(n_t)$ and let $k_t$ be the largest integer such that $\al \be g_{t-2}>k_tg_{t-1}$.  From  the above equation and \eqref{E-p}, we have $p_t=\al^2+\be+k_t$.  One can inductively prove that $\al g_{t-2}\le g_{t-1}$ for $t\ge 2$, and hence we must have $k_t<\be$.  From this it follows that $p(n)\le p_t\le \al^2+2\be-1$.  To show that an $n$ achieves this bound, we consider $n_3$.  Because $g_1=1$ and $g_2=\al$, we have that $k_3=\be-1$, from which it follows that $p_3=\al^2+2\be-1$.
	
	Regarding the statement of the second result, we note that the equality of $\ceil{\gam^2}$ and $\al^2+\be+\ceil{\al\be\gam^{-1}}$ follows from the fact that $\gam^2=\al^2+\be+\al\be\gam^{-1}$.  We claim that proving the second result is equivalent to proving that $k_t$ as defined above is equal to $\ceil{\al\be\gam^{-1}}-1$ for all sufficiently large $t$.  Indeed, the second result is equivalent to the existence of infinitely many $n$ with $p(n)\ge \al^2+\be+\ceil{\al\be\gam^{-1}}-1$ and only finitely many $n$ with $p(n)\ge \al^2+\be+\ceil{\al\be\gam^{-1}}$.  Because there are only finitely many $n$ with $t(n)=t$ for any fixed $t$, having infinitely many $n$ with $p(n)\ge p$ is equivalent to having infinitely many $t$ such that $p(n)\ge p$ for some $n$ with $t(n)=t$, and this is equivalent to having infinitely many $t$ such that $p_t\ge p$ since $p_t\ge p(n)$ for all $n$ with $t(n)=t$.  Because $p_t=\al^2+\be+k_t$, we are left with proving that $k_t= \ceil{\al\be \gam^{-1}}-1$ for all sufficiently large $t$.
	
	Recall that $k_t$ is the largest integer satisfying $\al\be g_{t-2}>k_tg_{t-1}$.  By using the closed form of $g_t$, this is equivalent to $\gam^{t-2}(\al\be-\gam k_t )-\lam^{t-2}(\al\be-\lam k_t)>0$.  Because $\lam^t=o(\gam^t)$, for $t$ sufficiently large this quantity will be positive if and only if $\al\be>k_t\gam$.  By the maximality of $k_t$, we conclude that for $t$ sufficiently large we have $k_t=\ceil{\al \be \gam^{-1}}-1$ as desired.
	
	The final result follows from the first two after verifying that having $(\be-1)>\al\be\gam^{-1}$ for reals $\al,\be$ with $\al\ge 1$ is equivalent to having $\be<\al+1$, and since we are interested in the case where $\al,\be$ are integers, this is equivalent to having $\be\le \al$.
\end{proof}

We now prove several results that will be of use to us in later sections and which are also of independent interest.  The first is an implicit bound on $s(n)$ given $n$.

\begin{lem}\label{L-Chicken}
	Let $s\ge 2$ be an integer.  If $n> \be g_{s-1}g_{s-2}$, then $s(n)\ge s$.
\end{lem}
\begin{proof}
	The result is trivial if $s=2$, so assume $s>2$.  Note that this implies $\gcd(g_{s-1},\be g_{s-2})=1$ by Lemma~\ref{L-FibLike}(b).  By Corollary~\ref{C-Form}, $s(n)$ is the largest integer such that there exists some $a,b\ge 1$ with $n= ag_{s(n)-1}+\be bg_{s(n)-2}$.  Thus it will be enough to show that there exist positive $a,b$ such that $n=ag_{s-1}+\be b g_{s-2}$, and this immediately follows from Lemma~\ref{L-Coin}.
\end{proof}
From this we can get a more explicit bound on $s(n)$.
\begin{prop}\label{P-Bound}
	For all $n$ with $s=s(n)>2$, we have
	\[
	\half \log_\gam(n)-1\le \s(n)\le \log_\gam(n)+2.
	\]
	
	Moreover, we have $s^{1,1}(n)\ge \half \log_\phi(n)+2$, where $\phi$ denotes the golden ratio.
\end{prop}
\begin{proof}
	Let $t=s-1$. From Theorem~\ref{T-GMain} and Lemma~\ref{L-Chicken}, we have \[g_{t}+\be g_{t-1}\le n \le \be g_{t}g_{t+1}\]  We wish to turn these bounds for $n$ into bounds for $t$.
	
	By using the closed form for $g_t$, $\lam\gam^{-1}=-\be\gam^{-2}$, and $\gam\ge 1$; we find 
	\begin{align*}
	n\ge g_t+\be g_{t-1}&=\rec{\gam+\be \gam^{-1}}(\gam^t-\lam^t+\be \gam^{t-1}+\be \lam^{t-1})\\&=\f{\gam^t}{\gam+\be \gam^{-1}}(1-(-\be \gam^{-2})^t+\be\gam^{-1}+\be(-\be \gam^{-2})^{t-1}\gam^{-1})\\&=\f{\gam^t}{\gam+\be \gam^{-1}}(1+\be \gam^{-1}-(-\be \gam^{-2})^t(1-\gam))\\ &\ge \f{\gam^t}{\gam+\be \gam^{-1}}(1+\be \gam^{-1}-(\be \gam^{-2})^t(\gam-1)).
	\end{align*}
	Note that by definition, $\gam^2\ge (\gam-\al/2)^2= \rec{4}(\al^2+4\be)\ge \be$, and in particular $\be^{t-1}\le \gam^{2t-2}$.  By using this and $t\ge 1$, we find that the above quantity is at least
	\begin{align*}
	\f{\gam^t}{\gam+\be \gam^{-1}}(1+\be\gam^{-1}-\be \gam^{-2}(\gam-1))=\f{\gam^t}{\gam+\be \gam^{-1}}(1+\be \gam^{-2})=\gam^{t-1}.
	\end{align*}
	Thus $n\ge \gam^{t-1}$, which implies $\log_\gam n\ge t-1$. Plugging in $t=s-1$ gives our desired upper bound on $s$.
	
	To establish a lower bound, we observe that $n\le \be g_{t+1}^2$ since $t\ge 2$.  Because $g_{t+1}=\rec{\gam-\lam}(\gam^{t+1}-\lam^{t+1})$ with $\lam<0$, we have $g_{t+1}\le \rec{\gam-\lam} \gam^{t+1}$ if $t$ is odd, and if $t$ is even we have $g_{t+1}\le g_{t+2}\le \rec{\gam-\lam} \gam^{t+2}$.  Further observe that $(\gam-\lam)^2=\al^2+4\be\ge \be$.  In total we conclude that 
	\begin{align*}
	\log_\gam n\le \log_\gam \be g_{t+1}^2\le \log_\gam \f{\be \gam^{2t+4}}{(\gam-\lam)^2}\le \log_\gam \gam^{2t+4}=2t+4, \nonumber
	\end{align*}
	and plugging in $t=s-1$ gives the desired general lower bound for $s$.
	
	For $\al=\be=1$, we use the more familiar notation $f_k=g_k$.  By using the closed form for $f_k$ and that $t\ge 2$, we find that 
	\begin{align*}
	n&\le f_tf_{t+1}=\rec{5}(\phi^t-(-\phi^{-1})^t)(\phi^{t+1}-(-\phi^{-1})^{t+1})\\&=\rec{5}\phi^{2t+1}(1-(-1)^t\phi^{-2t})(1+(-1)^t\phi^{-2t-2}).
	\end{align*}
	Because $\phi>1$ and $2t,2t+2\ge 0$, we have \[(1-(-1)^t\phi^{-2t})(1+(-1)^t\phi^{-2t-2})\le (1+\phi^{-2t})(1-\phi^{-2t-2}),\] and this is then maximized when $t$ is as small as possible.  Because $t\ge 2$, we have $n\le \rec{5}\phi^{2t+1}(1+\phi^{-4})(1-\phi^{-6})$.  Taking logarithms on both sides gives $\log_\phi n\ge 2t-2$, and plugging in $t=s-1$ gives the desired result.
\end{proof}

We next consider an algorithm for efficiently finding $n$-slow walks.  We note that this is essentially a generalization of the algorithm given in \cite{JK}.

\begin{prop}\label{P-Alg}
	For any fixed $\al,\be$, there exists an algorithm that runs in $O((\log n)^2)$ time which takes as input $n$ and returns all $(n,\al,\be)$-good pairs.
\end{prop}
\begin{proof}
	If $n\le \al\be$ then one can determine all of this information in $O(1)$ time, so assume $n>\al\be$.  By Lemma~\ref{L-Chicken}, this implies that $s(n)>2$.  By Theorem~\ref{T-GMain} it is enough to determine $a(n),\ b(n)$, and $t(n)$. 
	
	Set $t=0$ and compute $g_0,g_1$.  Assume that one has computed $g_{t},g_{t-1}$ and that the list of $n$-good pairs have not yet been found.  Set $t:=t+1$ and then $g_t:=\al g_{t-1}+\be g_{t-2}$ in $O(1)$ time.  Compute $b:= n(\be g_{t-1})^{-1} \mod g_t$, which can be done in $O(\log g_t)$ time by using the extended Euclidean algorithm \cite{C}, and choose the representative of $b$ so that $1\le b\le g_t$.  Set $a:=(n-\be g_{t-1}b)/g_t$.  Note that $a$ and $b$ are the unique integers with $ag_t+\be bg_{t-1}=n$ and $1\le b\le g_t$ given $t$.  If $(a,b)$ satisfies the remaining conditions of Theorem~\ref{T-GMain} (which can be checked in $O(1)$ time), then we conclude that $a(n)=a,\ b(n)=b,\ t(n)=t$ by the uniqueness of these integers.  Otherwise it must be that $t(n)\ne t$.
	
	One continues through the above procedure until it eventually computes $a(n),\ b(n),$ and $t(n)$.  Checking each value of $t$ takes at most $O(\log n)$ time, and because $t(n)\le \log_\gam(n)+1$ by Proposition~\ref{P-Bound}, we need to check at most $O(\log n)$ values of $t$.  We conclude the result.
\end{proof}

This algorithm can be improved when $\be=1$.
\begin{prop}
	For any fixed $\al$, there exists an algorithm that runs in $O(\log n)$ time which takes as input $n$ and returns all $(n,\al,1)$-good pairs.
\end{prop}
\begin{proof}
	By Corollary~\ref{C-alMain}, we know that $w_{s(n)+1}(b(n),a(n))=\floor{\gam n}$ or $\ceil{\gam n}$.  Motivated by this, we compute the two sequences $u_k,v_k$ defined by $u_1=\floor{\gam n},\ u_2=n,\ u_{k+2}=u_k-\al u_{k+1}$ for $k>2$; and $v_1=\ceil{\gam n},\ v_2=n,\ v_{k+2}=v_k-\al v_{k+1}$ for $k>2$.  The point with this is that if, say, $w_{s(n)+1}(b(n),a(n))=\floor{\gam n}$, the terms of $u_k$ will be exactly the terms of $w_k(b(n),a(n))$ but written in reverse order.  Thus by Corollary~\ref{C-alMain}, exactly one of these sequences $u_k,v_k$ will be the reverse of the sequence we are looking for.
	
	With this in mind, the algorithm works by computing each of these two sequences until the terms become non-positive.  Whichever sequence has more terms must correspond to $w_k(b(n),a(n))$.  In particular, its last two positive numbers will be $a(n)$ and $b(n)$, and from this we can find every $(n,\al,1)$-good pair.  It takes $O(1)$ time to compute each term in each of the sequences, and the longest sequence has length at most $O(\log n)$ by Proposition~\ref{P-Bound}, so we conclude the result.

\end{proof}
\section{Densities}\label{S-Den}
\subsection{A General Method}
A number of density results related to slow (1,1)-walks were proven in \cite{CGS}.  The same proof ideas easily generalize to $(\al,1)$-walks for any fixed $\al$, though the computations become messier. However, it is not immediately obvious how to apply these ideas to $(\al,\be)$-walks for $\be>1$ due to the divisibility condition of Theorem~\ref{T-GMain}.  Before proving Theorem~\ref{T-Den}, we first present a general method that can be used to get around this obstacle.

We will say that a pair $(a,b)$ is $t$-divisible if $a-\al b-\ell g_{t+1}$ is not a positive multiple of $\be$ for any $\ell\ge 0$.  Let $a_{b,t}^-$ and $a_{b,t}^+$ be real numbers depending on $b$ and $t$.  We will say that a set of positive integers $S$ is an interval set with respect to $a_{b,t}^-$ and $a_{b,t}^+$ if for all integers $t\ge 2$ and $1\le b\le g_t$ we have {\footnotesize\begin{equation}\label{E-Int}\{m\in S:b(m)=b,t(m)=t\}=\{ag_t+\be bg_{t-1}:a_{b,t}^-< a\le a_{b,t}^+,(a,b)\tr{ is }t\tr{-divisible}\}.\end{equation}}  By Theorem~\ref{T-GMain}, this essentially says that $S$ can be defined as the set of $m$ whose $a(m)$ lies in some interval depending only on $b(m)$ and $t(m)$. We note that we do not require $a_{b,t}^-$ and $a_{b,t}^+$ to be integers.  
	
All of the density results of \cite{CGS} were for interval sets or the complements of interval sets.  The main interval set of interest to us will be the following.

\begin{lem}\label{L-Sp}
	Let $S_p$ be the set of integers with $p(n)>p$.  Then $S_p$ is an interval set with respect to $a^-_{b,t}=\be p g_{t-1}$ and $a^+_{b,t}=(\be-1)g_{t+1}+\al b$.
\end{lem}
\begin{proof}
	Fix $t\ge 2$ and $1\le b\le g_t$.  With $S=S_p$, let $L$ denote the set defined on the left-hand set of \eqref{E-Int} and let $R$ be the right-hand set.  We wish to show that $L=R$.
	
	If $m=ag_t+\be bg_{t-1}$ is an element of $R$, then by Theorem~\ref{T-GMain} we have $a(m)=a$, $b(m)=b$, and $t(m)=t$.  Because $a(m)>\be p g_{t-1}$, we have $n\in S_p$ by Corollary~\ref{C-Pairs}, and hence $m\in L$. Conversely, if $m\in S_p$ with $t(m)=t$ and $b(m)=b$, then by Corollary~\ref{C-Pairs} and Theorem~\ref{T-GMain} we have $a^-_{b,t}<a(m)\le a_{b,t}^+$, $n=a(m)g_t+\be bg_{t-1}$, and that $(a(m),b)$ is $t$-divisible.  Thus $m\in R$ and we conclude the result.
\end{proof}

We now wish to refine our interval sets to deal more precisely with the divisibility condition.  For $0\le q\le \be-1$ and $t\ge 2$, we will say that a pair $(a,b)$ is $(q,t)$-good if \[1\le b\le g_t,\] \[\max\{0,(q-1)g_{t+1}+\al b\}< a\le q g_{t+1}+\al b.\]
Given a set of positive integers $S$, we define \[S(n,q,t)=\{m\in S:m\le n\ t(m)=t,\ (a(m),b(m))\tr{ is }(q,t)\tr{-good}\}.\]  Note that by Theorem~\ref{T-GMain}, every $m\in S$ with $m\le n$ lies in exactly one $S(n,q,t)$ set.

\begin{lem}\label{L-Snq}
	If $S$ is an interval set with respect to $a_{b,t}^-$ and $a_{b,t}^+$, then $S(n,q,t)$ is an interval set with respect to $\hat{a}_{b,t}^-$ and $\hat{a}_{b,t}^+$, where 
	\begin{align*}\hat{a}_{b,t}^-&=\max\{0,a_{b,t}^-,(q-1)g_{t+1}+\al b\},\ \\ \hat{a}_{b,t}^+&=\min\{a_{b,t}^+,q g_{t+1}+\al b, (n-\be bg_{t-1})g_t^{-1}\}.\end{align*}
\end{lem}
We note that $\hat{a}_{b,t}^-$ and $\hat{a}_{b,t}^+$ technically depend on $n$ and $q$, but we suppress this notation whenever $n$ and $q$ are understood.
\begin{proof}
	Fix $t\ge 2$ and $1\le b\le g_t$.  Let $L$ denote the set defined on the left-hand side of \eqref{E-Int}, with us now using $S(n,q,t)$ instead of $S$, and let $R$ be the set on the right-hand side.  We wish to show that $L=R$.
	
	For any $m\in L$, we have that $a_{b,t}^-<a(m)\le a_{b,t}^+$ since $m\in S$, that $\max\{0,(q-1)g_{t+1}+\al b\}<a(m)\le q g_{t+1}+\al b$ since $(a(m),b)$ is $(q,t)$-good, and that $a(m)g_t+\be bg_{t-1}\le n$ since $m\le n$.  Thus $m\in R$.  Conversely, if $m\in R$ then $b(m)=b$ and $t(m)=t$.  Because $S$ is an interval set and $a_{b,t}^-<a(n)\le a_{b,t}^+$, we have $m\in S$, and the other inequalities imply that $m\in S(n,q,t)$, so $m\in L$ and we conclude the result.
\end{proof}

When $\hat{a}_{b,t}^-$ and $\hat{a}_{b,t}^+$ are as in Lemma~\ref{L-Snq} and clear from context, we define \[T(n,q,t)=\{(a,b)\in \Z^2:\hat{a}_{b,t}^-<a\le \hat{a}_{b,t}^+,\ 1\le b\le g_t\}.\] 
Intuitively, $T(n,q,t)$ is obtained by taking $S(n,q,t)$, looking at $m$-good pairs instead of the numbers $m\in S(n,q,t)$, and then forgetting about the divisibility condition.  More precisely, we have the following.

\begin{lem}\label{L-Den}
	Let $S$ be an interval set with respect to $a_{b,t}^-$ and $a_{b,t}^+$.  Then for all $0\le q\le \be-1$ and $t\ge 2$, we have \[\l||S(n,q,t)|-\f{\be-q}{\be}|T(n,q,t)|\r|\le \be \min\{g_{t},n/g_{t-1}\}.\]
\end{lem}
\begin{proof}
	Let $T^*(n,q,t)\sub T(n,q,t)$ denote the subset of pairs which are $t$-divisible.  We claim that $|S(n,q,t)|=|T^*(n,q,t)|$.  Indeed, consider the map sending $m\in S(n,q,t)$ to $(a(m),b(m))$.  Because $S(n,q,t)$ is interval with respect to $\hat{a}_{b,t}^-$ and $\hat{a}_{b,t}^+$, this pair is in $T(n,q,t)$.  By Theorem~\ref{T-GMain} this pair is also $t$-divisible, so $(a(m),b(m))\in T^*(n,q,t)$.  Conversely, each $(a,b)\in T^*(n,q,t)$ can only be mapped to by $m=ag_t+\be bg_{t-1}$, which is in $S(n,q,t)$ by the interval condition. Thus this map defines a bijection between these two sets, proving the claim.
	
	We now wish to estimate $|T^*(n,q,t)|$.  To this end, let $1\le b\le g_t$ be fixed.  Let $T_b(n,q,t)=\{(a',b')\in T(n,q,t):b'=b\}$ and $T^*_b(n,q,t)\sub T_b(n,q,t)$ be the set of pairs which are $t$-divisible. Define \[R_b=\{-(\al b+\ell g_{t+1})\tr{ mod } \be:0\le \ell< q\}.\]  Observe that $R_b$ consists of $q$ distinct values since $q<\be$ and $\gcd(g_{t+1},\be)=1$ by Lemma~\ref{L-FibLike}.
	
	By construction, every element of $T(n,q,t)$ is $(q,t)$-good. This implies that any $(a,b)\in T_b(n,q,t)$ has $a-\al b-\ell g_{t+1}\le 0$ if $\ell\ge q$, and otherwise this quantity is positive.  Thus $(a,b)$ will be in $T^*_b(n,q,t)$ if and only if $a$ modulo $\be$ is any of the $\be-q$ values not in $R_b$.  We claim that roughly a $\f{\be-q}{\be}$ fraction of pairs in $T_b(n,q,t)$ will satisfy this condition. More precisely, we claim that
	\[
	\floor{\f{\be-q}{\be}|T_b(n,q,t)|}\le |T^*_b(n,q,t)|\le \floor{\f{\be-q}{\be}|T_b(n,q,t)|}+\be-1.
	\]
	Indeed, one can break up the elements of $T_b(n,q,t)$ into $\floor{\rec{\be} |T_b(n,q,t)|}$ disjoint sets of the form $\{(a,b),(a+1,b),\ldots,(a+\be-1,b)\}$, together with one remaining set of size at most $\be-1$.  Each of these disjoint sets of size $\be$ will contain exactly $\be-q$ pairs $(a,b)$ with $a\mod \be$ not in $R_b$, and the remaining set could have anywhere from 0 to $\be-1$ elements with this property.  This proves the claim.
	
	Since we only consider $(a,b)$ that are $(q,t)$-good, we only consider $b$ with $1\le b\le g_t$.  Since we also require $ag_t+\be b g_{t-1}\le n$, we also have $b\le n/g_{t-1}$.  Thus by summing the above bound over all $1\le b\le \min\{g_t,n/g_{t-1}\}$, we find
	{\footnotesize \[
	\f{\be-q}{\be}|T(n,q,t)|-\min\{g_t,n/g_{t-1}\}\le |T^*(n,q,t)|\le \f{\be-q}{\be}|T(n,q,t)|+(\be-1)\min\{g_t,n/g_{t-1}\}.
	\]}
	By using $1,\be-1\le \be$, we conclude the result.
\end{proof}

The $T(n,q,t)$ sets are relatively easy to work with since they have no divisibility conditions.  In particular, the cardinality of these sets is exactly $\max\{0,a_{b,t}^+-a_{b,t}^-\}\cdot g_t$.  One often needs to further refine $T(n,q,t)$ by replacing its bounds with asymptotic estimates, and we will see an example of this in the proof of Theorem~\ref{T-Den}.  For now we record a general result that can be used in computing densities of interval sets.

\begin{cor}\label{C-Den}
	If $S$ is an interval set, then
	\[
	\l||S\cap [n]|-\sum_{q,t} \f{\be-q}{\be} |T(n,q,t)|\r|=O(\sqrt{n}).
	\]
\end{cor}
\begin{proof}
	The sum $\sum_{t,q} |S(n,q,t)|$ counts every element of $|S|\cap[n]$ exactly once, except for those $m\in S$ with $t(m)=1$.  By Lemma~\ref{L-Chicken} there are at most $\al\be$ such exceptions.  Thus by Lemma~\ref{L-Den}, the triangle inequality, and the fact that $g_t=\Theta(\gam^t)$, we find that
	\begin{align*}
	\l||S\cap[n]|-\sum_{q,t} \f{\be-q}{\be} |T(n,q,t)|\r|&\le 
	\al\be+\sum_{q,t} \l||S(n,q,t)|-\f{\be-q}{\be}|T(n,q,t)|\r|\\ &\le O(1)+\sum_{t:g_tg_{t-1}\le n} O(\gam^t)+\sum_{t:g_tg_{t-1}\ge n} O(n\gam^{-t})\\ &=O(\sqrt{n}),
	\end{align*}
	where we used that the largest terms of the geometric sums are $O(\sqrt{n})$.
\end{proof}
\subsection{Proving Theorem~\ref{T-Den}}
Recall that we define $n_{c,r}:=\floor{\f{c\be}{(\gam-\lam)^2} \gam^{2r+1}}$.  Before proving Theorem~\ref{T-Den}, we first show that most $n$ with $p(n)$ large and $n\le n_{c,r}$ have $t(n)\le r$, and we also establish a more precise cutoff for when this occurs.

\begin{lem}\label{L-Pairs}
	Let $r\ge 2$.  For any $p\ge \be$, there are at most $O(\be^r)$ many $n$ with $p(n)>p$, $n\le \f{\be (p-\be+1) }{\al(\gam-\lam)^2}\gam^{2r+2}$, and $t(n)>r$.  Further, all $n$ with $p(n)>\be$ and $n\le \be g_rg_{r+1}$ satisfy $t(n)\le r$.
\end{lem}

\begin{proof}
	Fix some $n$ with $p(n)>p\ge \be$, and let $a=a(n),\ b=b(n)$, and $t=t(n)$.  By Theorem~\ref{T-GMain} and Corollary~\ref{C-Pairs} we have \[\be p g_{t-1}< a\le (\be-1)g_{t+1}+\al b.\]  Using these lower bounds on $a$ and $b$, we find \begin{align*}n&=ag_{t}+\be bg_{t-1}\nonumber\\ &> \be p g_{t-1}g_{t}+\l(\f{\be p}{\al}\be g_{t-1}-\f{\be-1}{\al}\be g_{t+1}\r)g_{t-1}\\ &=\f{\be p}{\al}g_{t-1}(\al g_t+\be g_{t-1})-\f{(\be-1)\be}{\al}g_{t-1}g_{t+1}\\&=\f{\be p}{\al} g_{t-1}g_{t+1}-\f{(\be-1)\be}{\al}g_{t-1}g_{t+1}=\f{\be (p-\be+1)}{\al} g_{t-1}g_{t+1} .\end{align*} 
	Now assume $t>r$, so that we conclude $n>\f{\be(p-\be+1)}{\al} g_rg_{r+2}$.  Our first result follows by observing that this implies that $n$ is at least \[\f{\be (p-\be+1)}{\al} g_rg_{r+2}=\f{\be(p-\be+1) }{\al(\gam-\lam)^2}\gam^{2r+2}+O(\be^r).\]  For the second result, note that $g_{r+2}=\al g_{r+1}+\be g_{r}\ge  \al g_{r+1}$ since $r\ge 2$, and since $p\ge\be$ we conclude that $n>\be g_rg_{r+1}$.
\end{proof}

\begin{proof}[Proof of Theorem~\ref{T-Den}]
	Define $n_r=\be g_r g_{r+1}$.  By Lemma~\ref{L-Chicken}, the only elements in $S_p\cap [n_{c,r}]$ that are not in $S_p\cap [n_r]$ are $n$ with $n\le n_{c,r}$, $p(n)>p\ge \be$, and $t(n)>r$.  By Lemma~\ref{L-Pairs} there are at most $O(\be^r)$ such $n$.  The number of elements in $S_p\cap [n_r]$ that are not in $S_p\cap [n_{c,r}]$ is at most $n_r-n_{c,r}\le n_r-n_{1,r}=O(\be^r)$.  In total we conclude \begin{equation} \label{E-Red}||S_p\cap[n_r]|-|S_p\cap[n_{c,r}]||=O(\be^r),\end{equation} so it will be enough for us to estimate the size of $S_p\cap [n_r]$.
	
	By Lemma~\ref{L-Sp}, $S_p$ is an interval set with respect to $a_{b,t}^-=\be p g_{t-1}$ and $a_{b,t}^+=(\be-1)g_{t+1}+\al b$.   With $S=S_p$, let $S(n_r,q,t)$ and $T(n_r,q,t)$ be the sets as defined before Lemmas~\ref{L-Snq} and \ref{L-Den}.  Our goal at this point is, for each $q$ and $t$, to estimate the sizes of either $S(n_r,q,t)$ or $T(n_r,q,t)$.
	
	By Lemma~\ref{L-Pairs}, there exists no $m\in S_p\cap[n_r]$ with $t(m)>r$.  Thus $|S(n_r,q,t)|=0$ for $t>r$.  Note that for all $t\le r$, $q\le \be-1$, and $b\le g_t$; we have $qg_tg_{t+1}+bg_{t+1}\le \be g_rg_{r+1}=n_r$.  This inequality is equivalent to $qg_{t+1}+\al b\le (n_r-\be bg_{t-1})g_t^{-1}$.  Thus for $t\le r$ we have $\hat{a}_{b,t}^+=qg_{t+1}+\al b$, and we recall that $\hat{a}_{b,t}^-=\max\{0,\be p g_{t-1},(q-1)g_{t+1}+\al b,\}$.
	
	Recall that $d$ is the smallest integer such that $\del:=\be\gam^{-1}p-\gam d\le \al$.   We first establish the range that $d$ and $\del$ can lie in.
	\begin{claim}\label{Cl-d}
		We have $0\le d\le \be-1$ and $\al-\gam<\del\le \al$.
	\end{claim}
	\begin{proof}
		Because $p\ge 1$, we have 
		\[
		\be\gam^{-1}p+\gam\ge -\lam+\gam=\al,
		\] 
		so $d\ge 0$.  Because $p\le \ceil{\gam^2}-2<\gam^2-1$, we have
		\[
			\be \gam^{-1}p-(\be-1) \gam<\gam-\be \gam^{-1}=\al,
		\]
		where we used $\gam^2-\be=\al \gam$, so $d\le \be-1$.	We have $\del\le \al$ by definition.  If $\al-\gam \ge\del$, then $\be \gam^{-1}p-\gam(d-1)\le \al$, a contradiction to the definition of $d$.
	\end{proof}
	
	Now that we understand the range that $d$ can take on, we turn to estimating $|T(n_r,q,t)|$.  First consider any $q$ with $0\le q<d$ (if such a $q$ exists).  In this case we have for $b\le g_t$ and $t\le r$ that
	\begin{align*}\hat{a}_{b,t}^--\hat{a}_{b,t}^+&\ge \be p g_{t-1}-qg_{t+1}-\al b= \f{\be p \gam^{-1}-q\gam }{\gam-\lam}\gam^t-\al b+O(\lam^t)\\ &\ge \f{\be p\gam^{-1}-(d-1)\gam-\al }{\gam-\lam}\gam^{t}+o(\gam^t)\\&=\f{\gam+\del-\al}{\gam-\lam}\gam^{t}+o(\gam^t).\end{align*}

	By the previous claim, $\gam+\del-\al>0$. Thus for any $0\le q<d$ and $t$ sufficiently larger than some constant depending only on $p$, $\al$, and $\be$; we have $\hat{a}^-_{b,t}>\hat{a}_{b,t}^+$, and hence $|T(n_r,q,t)|=\max\{0,a_{b,t}^+-a_{b,t}^-\}\cdot g_t=0$ for such $q$ and $t$.  Thus in total we have
	\begin{align}\label{E-Tsmall}
	\sum_{0\le q<d} \sum_t |T(n_r,q,t)|=O(1).
	\end{align}
	
	Observe that $d+1\ge \be \gam^{-2}p-\al \gam^{-1}+1$.  Thus for $q$ with $0<d+1<q<\be$, we have \begin{align*}(q-1)g_{t+1}+\al b&\ge (d+1)g_{t+1}\ge \f{\be \gam^{-2}p-\al \gam^{-1}+1}{\gam-\lam}\gam^{t+1}+O(\lam^t)\\ &= (\be p -\al\gam+\gam^2) g_{t-1}+O(\lam^t)>\be p  g_{t-1}+O(\lam^t),\end{align*}  where we used that $\gam>\al$ for all $\al,\be$.  Thus for any $d+1<q<\be$ we have $\hat{a}_{b,t}^-=(q-1)g_{t+1}+\al b+O(\lam^t)$, and hence {\small\begin{equation}\label{E-D1}|T(n_r,q,t)|=\max\{0,a_{b,t}^+-a_{b,t}^-\}\cdot g_t=g_{t+1}g_t+O(\be^t)=\rec{(\gam-\lam)^2}\gam^{2t+1}+O(\be^t).\end{equation}}
	
	It remains to deal with the cases $q=d$  and $q=d+1$.  In these cases it will be easier to work with an asymptotic version of $T(n_r,q,t)$. To this end, define \begin{align*}\tilde{a}_{b,t}^-&=\max\{0,\f{q-1}{\gam-\lam} \gam^{t+1}+\al b,\f{\be p }{\gam-\lam}\gam^{t-1}\},\\ \tilde{a}_{b,t}^+&=\f{q}{\gam-\lam} \gam^{t+1}+\al b.\end{align*} Technically $\tilde{a}_{b,t}^\pm$ depend on $q$, but we suppress this from the notation.  Let $U(q,t)=\{(a,b)\in \Z^2:\tilde{a}_{b,t}^-<a\le \tilde{a}_{b,t}^+,\ 0\le b\le \rec{\gam-\lam}\gam^t\}$. 
	
	\begin{claim}\label{Cl-Den}
		For all $t\le r$ we have $|T(n_r,q,t)-U(q,t)|=O(\gam^t+\be^t)$.
	\end{claim}
	\begin{proof}
		Let $U'(q,t)=\{(a,b)\in \Z^2:\hat{a}_{b,t}^-<a\le \hat{a}_{b,t}^+,\ 0\le b\le \rec{\gam-\lam}\gam^t\}$.  Any pair $(a,b)\in U'(q,t)$ that is not in $T(n_r,q,t)$ must have either $b=0$ or $g_t<b\le \rec{\gam-\lam}\gam^t$.  Note that the number of $b$ in this latter range is at most $\ceil{|g_t-\rec{\gam-\lam}\gam^t|}\le1+\rec{\gam-\lam}(-\lam)^t$.  For these $b$, there are at most $\hat{a}_{b,t}^+-\hat{a}_{b,t}^-+1\le g_{t+1}+1=O(\gam^t)$ values that $a$ can take on. Thus $U'(q,t)$ has at most $O((1+\lam^t)\gam^t)=O(\gam^t+\be^t)$ more pairs than $T(n_r,q,t)$.  Using similar logic, we find that the same bound holds for the number of pairs in $T(n_r,q,t)$ but not $U'(q,t)$, and hence $||T(n_r,q,t)|-|U'(q,t)||=O(\gam^t+\be^t)$.
		
		Similarly, note that $\ceil{|\hat{a}_{b,t}^+-\tilde{a}_{b,t}^+|}$ and $\ceil{|\hat{a}_{b,t}^--\tilde{a}_{b,t}^-|}$ are $O(1+\lam^t)$.  Thus for any fixed $b$ there are at most $O(1+\lam^t)$ values of $a$ such that $(a,b)\in U(q,t)$ and $(a,b)\notin U'(q,t)$, and vice versa.  Thus we have $|U'(q,t)-U(q,t)|=O(\gam^t+\be^t)$, and from this the result follows.

	\end{proof}
	
	From this point forward we will omit floors and ceilings whenever they do not significantly affect our computations.  We first consider $U(d,t)$.  One can verify from the definition of $d$ that for $q=d$ we have $\tilde{a}_{b,t}^-=\f{\be p }{\gam-\lam}\gam^{t-1}$ for all $b\le \rec{\gam-\lam}\gam^t$.  Thus \[\tilde{a}_{b,t}^+-\tilde{a}_{b,t}^-=\f{\gam d-\be \gam^{-1}p}{\gam-\lam }\gam^t+\al b=\f{-\del}{\gam-\lam}\gam^t+\al b.\]  Hence for each fixed value of $b$, the number of $(a,b)$ we have in $U(d,t)$ will be $\max\{0,\f{-\del}{\gam-\lam}\gam^t+\al b\}$.  Note that this value is non-zero precisely when $b\ge \f{\al^{-1}\del}{\gam-\lam}\gam^t:=b'$.  We conclude that 
	\begin{align}\label{E-Use}
	|U(d,t)|&=\sum_{b=b'}^{\gam^t/(\gam-\lam)}\l(\f{-\del}{\gam-\lam}\gam^t+\al b\r).\end{align}
	Using the general fact that $\sum_{i=x}^y z+i=z(y-x+1)+\half(y^2-x^2)+O(y)$, we find that this equals
	\begin{align}\label{E-D2}\f{-\del+\al^{-1}\del^2}{(\gam-\lam)^2}\gam^{2t}+ \f{\al-\al^{-1}\del^2}{2(\gam-\lam)^2}\gam^{2t}+O(\gam^t)=\f{\al-2\del+\al^{-1}\del^2}{2(\gam-\lam)^2}\gam^{2t}+O(\gam^t).
	\end{align}
	
	Now consider $U(d+1,t)$.  Again let $b':=\f{\al^{-1}\del }{\gam-\lam}\gam^t$.  One can verify that we have $\tilde{a}_{b,t}^-=\f{d}{\gam-\lam}\gam^{t+1}+\al b$ if $b\le b'$ and $\tilde{a}_{b,t}^-=\f{\be p }{\gam-\lam}\gam^{t-1}$ otherwise. First we count the pairs $(a,b)$ with  $b\le b'$. Because $\tilde{a}_{b,t}^-=\f{d}{\gam-\lam}\gam^{t+1}+\al b$, for each of the $b'$ possible values for $b$ in this range, there are exactly $\tilde{a}_{b,t}^+-\tilde{a}_{b,t}^-=\rec{\gam-\lam}\gam^{t+1}$ choices for $a$, so in total there are $\f{\al^{-1}\del \gam}{(\gam-\lam)^2}\gam^{2t}+O(\gam^t)$ many pairs with $b$ in this range.  For pairs with $b\ge b'$, we again take $\tilde{a}_{b,t}^+-\tilde{a}_{b,t}^-$ and sum over all $b$ to find that the number of such pairs is
	\begin{equation}\label{E-DInt}
	\sum_{b=b'}^{\gam^t/(\gam-\lam)} \f{\gam-\del}{\gam-\lam}\gam^t+\al b=\f{\al-2\del+\al^{-1}\del^2+2\gam-2\al^{-1}\del\gam}{2(\gam-\lam)^2}\gam^{2t}+O(\gam^t),
	\end{equation}
	where one can use that the summand of \eqref{E-DInt} is simply the summand of \eqref{E-Use} plus $\f{\gam}{\gam-\lam}\gam^t$ to quickly compute this sum.
	
	In total then we conclude that
	\begin{align}
	|U(d+1,t)|&=\f{\al^{-1}\del \gam}{(\gam-\lam)^2}\gam^{2t}+\f{\al-2\del+\al^{-1}\del^2+2\gam-2\al^{-1}\del\gam}{2(\gam-\lam)^2}\gam^{2t}+O(\gam^t)\nonumber\\ &=\f{\al-2\del+\al^{-1}\del^2+2\gam}{2(\gam-\lam)^2}\gam^{2t}+O(\gam^t).\label{E-D3}
	\end{align}
	
	Let $S(n_r,q)=\bigcup_{t\le r} S(n_r,q,t)$, and similarly define $T(n_r,q)$ and $U(q)$.  As mentioned earlier, $|S(n_r,q,t)|=0$ for $t>r$, so we have \[|S_p\cap[n_r]|=\sum_{q=0}^{\be-1} |S(n_r,q)|.\] By  Lemma~\ref{L-Den}, Claim~\ref{Cl-Den}, and \eqref{E-Tsmall}; this sum is equal to \begin{equation}\label{E-Den}\f{\be-d}{\be}|U(d)|+\f{\be-d-1}{\be}|U(d+1)|+\sum_{q=d+2}^{\be-1}\f{\be-q}{\be}|T(n_r,q)|+O(\gam^r+\be^r),\end{equation}
	where implicitly we used that the error term from Lemma~\ref{L-Den} is at most $\sum_q\sum_{t\le r} g_t=O(\gam^r)$.  We also recall that \eqref{E-Den} is equal to $|S_p\cap[n_{c,r}]|+O(\be^r)$ by \eqref{E-Red} (note the use of $n_{c,r}$ instead of $n_r$), which is ultimately what we wish to compute.  Note that by \eqref{E-D1}, for all $q$ with $d+2\le q\le \be-1$ we have
	\begin{align*} 
	|T(n_r,q)|=\sum_{t\le r} \rec{(\gam-\lam)^2} \gam^{2t+1}+O(\gam^t)=\f{\gam^3}{(\gam-\lam)^2(\gam^2-1)}\gam^{2r}+O(\gam^r).
	\end{align*}
	Similarly from \eqref{E-D2} we find
	\begin{align*}
	|U(d)|&=\f{\gam^2(\al-2\del+\al^{-1}\del^2)}{2(\gam-\lam)^2(\gam^2-1)}\gam^{2r}+O(\gam^r),\\\end{align*} and from \eqref{E-D3}, \begin{align*} |U(d+1)|&=\f{\gam^2(\al-2\del+\al^{-1}\del^2)}{2(\gam-\lam)^2(\gam^2-1)}\gam^{2r}+\f{\gam^3}{(\gam-\lam)^2(\gam^2-1)}\gam^{2r}+O(\gam^r).
	\end{align*}
	Plugging these values into \eqref{E-Den} and dividing by $n_{c,r}\sim \f{c}{(\gam-\lam)^2}\gam^{2r+1}$ gives the desired result.
\end{proof}

We note that in principle one can use these same methods to prove stronger results.  For example, one can prove density results for $p<\be$ or for $c>\gam/\al$. In these cases the condition $a g_t+\be bg_{t-1}\le n$ becomes non-trivial, which makes the computations messier.  We refer the reader to \cite{CGS} for examples of such results when $\al=\be=1$.

We now prove our corollaries to Theorem~\ref{T-Den}.

\begin{proof}[Proof of Corollary~\ref{C-pLarge}]
	By looking at the formula of Theorem~\ref{T-Den}, we see that we need to show $d=\be-1$ for these values of $p$.  Note that from Claim~\ref{Cl-d} we already know $d\le \be-1$, so it remains to show $d>\be-2$, or equivalently that $\be \gam^{-1} p-\gam (\be-2)-\al>0$.  This is equivalent to \begin{align*}
	p&>\gam^2 -2\be^{-1}\gam^2+\al\gam\be^{-1}=(1-\be^{-1})\gam^2+\be^{-1}\gam(\al-\gam)=(1-\be^{-1})\gam^2-1,
	\end{align*}
	where in the last equality we used $\al-\gam=\lam$ and $\be^{-1}\gam=-\lam^{-1}$.  Because $\floor{x}>x-1$, this result holds for $p\ge \floor{(1-\be^{-1})\gam^2}$ as desired.
\end{proof}

\begin{proof}[Proof of Corollary~\ref{C-be=1}]
	The first result follows by plugging $\be=1$ into Corollary~\ref{C-pLarge},  and then using $\gam^2-1=\al\gam$ and that $\al-2\gam^{-1}p+\al^{-1}\gam^{-2}p^2=(\al^{1/2}-\al^{-1/2}\gam^{-1}p)^2$.  The second result follows from $1-\al\gam^{-1}=\gam^{-2}$.
\end{proof}
\section{Slowest Slow Walks}\label{S-SS}
Before proving Theorem~\ref{T-Slowest}, we first consider a slightly more general setting.  We say that a set $T$ is valid if \[T\sub \{(\al,\be):\al,\be\ge1,\ \gcd(\al,\be)=1\}.\]  Given any valid $T$, define $\ss_T(n)=\max_{(\al,\be)\in T}\s(n)$ and $\S_T(n)=\{(\al,\be):\s(n)=\ss_T(n)\}$.  Thus $\ss(n)$ and $\S(n)$ correspond to the case when $T$ is every pair of relatively prime positive integers.  It will be of use to define $\Gamma_T=\min\{\gam_{\al,\be}:(\al,\be)\in T\}$, and for ease of notation we will simply write this as $\Gamma$ whenever $T$ is understood.

\begin{thm}\label{T-TSlowest}
	For any valid set $T$, there exists a finite set $R_T$ and number $n_T$ such that $\S_T(n)\sub R_T$ for all $n\ge n_T$.
\end{thm}
\begin{proof}
	Let $(\al',\be')$ be any pair in $T$ with  $\gam_{\al',\be'}=\GAM$.  For $n>\al'\be'$ we have $s^{\al',\be'}(n)>2$ by Lemma~\ref{L-Chicken}, and hence $\ss_T(n)>2$ for sufficiently large $n$.  If $(\al,\be)\in \S_T(n)$ and $\gam:=\gam_{\al,\be}$, then by applying Proposition~\ref{P-Bound} twice we find, for $n$ sufficiently large,  
	\[
	\log_{\gam} n+2\ge \s(n)=\ss_T(n)\ge s^{\al',\be'}(n)\ge \half \log_{\GAM} n-1> \quart \log_{\GAM} n+2.
	\]
	In particular this implies $\log_{\GAM}\gam<4$.  Since $\Gam$ is monotonically increasing in $\al$ and $\be$, the set $R'=\{(\al,\be):\log_{\GAM} \Gam<4\}$ is finite, so we conclude the result by taking $R_T=R'\cap T$.
\end{proof}

It is natural to ask how often a given $(\al,\be)$ lies in $\S_T(n)$.  To this end, for any set of pairs $S$, we define \[\I_T(S)=\{n:S\sub  \S_T(n)\},\ \E_T(S)=\{n:S=\S_T(n)\}.\]
For ease of notation we write $\I_T(\al,\be)$ and $\E_T(\al,\be)$ instead of $\I_T(\{(\al,\be)\})$ and $\E_T(\{(\al,\be)\})$.  We first provide an upper bound on the number of $n$ which have $(\al,\be)\in \S_T(n)$.

\begin{prop}\label{P-Slowest}
	Let $T$ be a valid set.  Given $(\al,\be)\in T$, let $\gam=\Gam$ and $c=\log_{\GAM}\gam$.  Then \[\l|\I_T(\al,\be)\cap[n]\r|=O(n^{2-c}).\]
\end{prop}

\begin{proof}
	Let $(\al',\be')$ be any pair in $T$ with $\gam_{\al',\be'}=\GAM$, let $\I=\I_T(\al,\be)$, and let $G_t=g_t^{\al',\be'}$.  For the rest of the proof let $g_t=\g_t$.
	
	Fix an integer $t$.  If $\t(m)=t$, then by Theorem~\ref{T-GMain} we have $m=ag_t+\be g_{t-1}$ for some $a,b\ge 1$.  If $m\in \I$, then we require $t^{\al',\be'}(m)\le t$.  In particular, by Lemma~\ref{L-Chicken} we require \[ag_t+\be g_{t-1}=m\le \be G_t G_{t-1}=O(\GAM^{2t}).\]  Because $g_t,g_{t-1}=\Om(\gam^t)$, we have \[a+b=O(\GAM^{2t}\gam^{-t})=O(\GAM^{(2-c)t}).\]  There are at most $O(\GAM^{(4-2c)t})$ positive $a,b$ satisfying this, and because $m$ is uniquely determined by $a$ and $b$ given $t$, we conclude that there are at most $O(\GAM^{(4-2c)t})$ integers $m$ with $\t(m)=t$ and $m\in \I$.
	
	Let $r$ be the smallest integer such that $n\le \be G_r G_{r-1}$, and note that $n=\Theta(\GAM^{2r})$.  By summing our above bound for all $t\le r$, we find \[|\I\cap [n]|=O(\GAM^{(4-2c)r})=O(n^{2-c}),\]
	as desired.
\end{proof}

We do not suspect that this bound is sharp in general.  However, it is strong enough to give the following.
\begin{cor}\label{C-Slowest}
	If $T$ is a valid set such that there exists a unique pair $(\al',\be')\in T$ with $\gam_{\al',\be'}=\GAM$, then \[n^{-1}|\E_T(\al',\be')\cap[n]|\sim 1.\]  That is, almost every $n$ has $\S_T(n)=\{(\al',\be')\}$.
\end{cor}
\begin{proof}
	Let $R_T$ be as in Theorem~\ref{T-TSlowest}.  For each pair $(\al,\be)\in R_T\sm  \{(\al',\be')\}$, we have $\gam_{\al,\be}>\GAM$ by assumption, so  $n^{-1}|\I_T(\al,\be)\cap[n]|=o(1)$ by Proposition~\ref{P-Slowest}.  Because $\E_T(S)\sub \I_T(\al,\be)$ whenever $(\al,\be)\in S$, this implies that $n^{-1}|\E_T(S)\cap[n]|=o(1)$ for each of the finitely many subsets $S\sub R_T$ with $S\ne \{(\al',\be')\}$.  Because every $m$ with $n_T\le m\le n$ has $m\in \E_T(S)$ for some $S\sub R_T$, we conclude the result.
\end{proof}

We now use the ideas developed in this section to prove Theorem~\ref{T-Slowest}.
\begin{proof}[Proof of Theorem~\ref{T-Slowest}]
	Because $n>1$, we have $s^{1,1}(n)>2$ by Lemma~\ref{L-Chicken}, and hence $\ss(n)>2$.  If $(\al,\be)\in \S(n)$, then by applying Proposition~\ref{P-Bound} twice we find 
	\[
		\log_{\Gam} n+2\ge \s(n)=\ss(n)\ge s^{1,1}(n)\ge \half \log_\phi n+2.
	\]
	For $n>1$, this inequality is equivalent to $\log_\phi \gam \le 2$.  Note that $\log_\phi \gam_{3,1}\ge 2.4$, $\log_\phi \gam_{2,2}\ge 2.08$, and $\log_\phi \gam_{1,5}\ge 2.1$.  Because $\gam_{\al,\be}$ is monotonic in $\al$ and $\be$, we conclude that no element not in $R$ can be in $\S(n)$.  In other words, $\S(n)\sub R$.
	
	The above result implies $\ss(n)=\max \{\s(n): (\al,\be)\in R\}$.  Because $R$ is a finite set, one can compute (in finite time) $\ss(n)$ and $\S(n)$ for any fixed $n$ by computing $\s(n)$ for each $(\al,\be)\in R$.  In particular, one can verify (using the algorithm of Proposition~\ref{P-Alg}) that \begin{align*}\S(32)=\{(1,1),(1,2)\},\ \S(40)&=\{(1,1),(1,3)\},\ \S(3363)=\{(1,1),(2,1)\},\\ \S(5307721328585529)&=\{(1,1),(1,4)\}.\end{align*} proving the second part of the theorem.  The third part of the theorem follows from Corollary~\ref{C-Slowest} and the monotonicity of $\gam_{\al,\be}$ in $\al$ and $\be$.
\end{proof}

We note that $5307721328585529=g^{1,4}_{39}+4g^{1,4}_{38}$, and in general it seems that $n$ of the form $n=\g_t+\be \g_t$ are more likely to have $(\al,\be)\in \S(n)$.

\section{Further Problems}\label{S-Con}
There are a number of questions about slow walks which remain to be answered.  For example, one could try and extend other results and problems from \cite{CGS} to $(\al,\be)$-walks in general.  There are also many questions one can ask about slowest slow walks.  Recall that for $T$ a set of relatively prime pairs, we define $\ss_T(n)=\max_{(\al,\be)\in T}\s(n)$, $\S_T(n)=\{(\al,\be):\s(n)=\ss_T(n)\}$, and $\GAM=\Gamma_T=\min\{\gam_{\al,\be}:(\al,\be)\in T\}$.  From now on we fix such a set of pairs $T$.  Define $R_T'=\{(\al,\be)\in T: \gam_{\al,\be}<\GAM^2\}$.

\begin{conj}
	There exist infinitely many $n$ with $(\al,\be)\in \S_T(n)$ if and only if $(\al,\be)\in R_T'$.
\end{conj}
We note that a sharper analysis of the proof of Theorem~\ref{T-TSlowest} shows that one can take $R_T\sub \{(\al,\be)\in T:\log_{\GAM} \gam_{\al,\be}\le 2\}$, so for the ``only if'' direction one need only verify the case $\gam_{\al,\be}=\GAM^2$. One possible direction to prove the ``if'' direction is the following.

\begin{conj}
	Let $(\al,\be)\in R_T'$.  There exist infinitely many $t$ such that $(\al,\be)\in \S_T(\g_t+\be \g_{t-1})$.
\end{conj}
We note that it was by looking at integers of the above type that we found an $n$ with $(1,4)\in \S(n)$.  Using this same approach, we were also able to verify that \begin{align*}\S(22619537)&=\{(2,1)\},\  \S(171)=\{(1,2)\},\ \S(11228332)=\{(1,3)\},\\ &\S(5000966512101628011743180761388223)=\{(1,4)\}.\end{align*}
Thus for every $(\al,\be)\in R$ we have $\S(n)=\{(\al,\be)\}$ for some $n$.  However, other than $n=1$ we have no examples where $|\S(n)|>2$. This discussion motivates the following questions.  Recall $\I_T(S)=\{n:S\sub  \S_T(n)\}$ and $ \E_T(S)=\{n:S=\S_T(n)\}$.

\begin{quest}
	Which sets of pairs $S$ have $|\E_T(S)|>0$? Which have $|\E_T(S)|=\infty$?
\end{quest}

In addition to determining the cardinality of these sets, it would be of interest to determine the growth rates of $|\I_T(S)\cap[n]|$ and $|\E_T(S)\cap [n]|$.

\begin{quest}
	Can one effectively bound $|\I_T(S)\cap[n]|$ and $|\E_T(S)\cap[n]|$ for various sets of pairs $S$?
\end{quest}

An interesting case is $T=\{(1,6),(2,3)\}$.  Note that $\gam_{1,6}=\gam_{2,3}=3$, so Corollary~\ref{C-Slowest} does not apply here.

\begin{quest}
	With $T=\{(1,6),(2,3)\}$, can one bound $|\E_T(1,6)\cap[n]|$ and $|\E_T(2,3)\cap[n]|$?
\end{quest}

We show some computational data regarding this question in the appendix.

In this paper we defined $w_k^{\al,\be}$ to follow a recurrence of depth two.  More generally, if $\boldsymbol{\al}=(\al_0,\ldots,\al_{r-1})$ is any tuple of positive relatively prime integers, one can define an $n$-slow $\boldsymbol{\al}$-walk to be a sequence  $w_k^{\boldsymbol{\al}}$ of positive integers satisfying $w_{k+r}^{\boldsymbol{\al}}=\sum_{i=0}^{r-1}\al_i w_{k+i}^{\boldsymbol{\al}}$ and which generates $n$ as slowly as possible.

\begin{quest}
	What can be said about $n$-slow $\boldsymbol{\al}$-walks?  In particular, what can be said about ``tribonacci walks'' which have $\boldsymbol{\al}=(1,1,1)$?
\end{quest}

Given these definitions, for any set $T$ of tuples of positive relatively prime integers, one could define $\ss_T(n)=\max \{s^{\boldsymbol{\al}}(n)+2-|\boldsymbol{\al}|:\boldsymbol{\al}\in T\}$ and $\S_T(n)=\{\boldsymbol{\al}\in T: s^{\boldsymbol{\al}}(n)=\ss_T(n)\}$.  Note that we subtract the length of $\boldsymbol{\al}$ in the definition of $\ss_T(n)$ to make this quantity finite, as otherwise one could always consider the tuple $(1,1,\ldots,1,n)$ with any number of 1's.  This issue could also be resolved by only considering $T$ consisting of tuples of bounded length.

\begin{quest}
	What can be said about $\ss_T(n)$ and $\S_T(n)$ when $T$ is a set of tuples?
\end{quest}

\section*{Appendix: Computational Data}

We first present results concerning densities of the sets $S_p$.  Throughout we denote the asymptotic densities that Theorem~\ref{T-Den} predicts by gray curves and the actual data by black dots.  We show the data curve for all $n_{c,r}$ with $1\le c\le \gam^2$ and some fixed value $r$, though we recall that Theorem~\ref{T-Den} only gives information in the range $1\le c\le (p-\be+1)\gam/\al$.

We first consider the case $\al=\be=1$.  The data we show corresponds to the case $r=12$ and $p=1$ in the statement of Theorem~\ref{T-Den}.  

\Graph{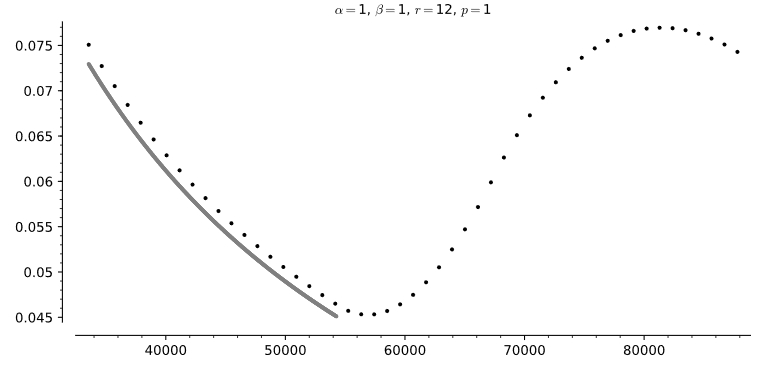}

We note that \cite{CGS} includes figures with this theory curve extended to all $1\le c\le \gam$.  

Next we show data for $\al=2,\be=1,r=7$ and $1\le p\le 4$.

\Graph{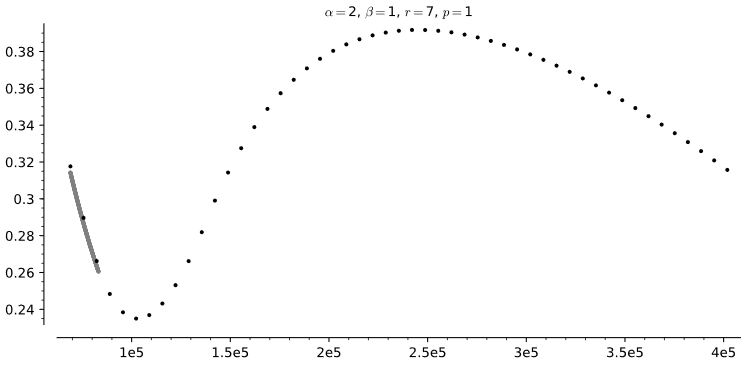}

\Graph{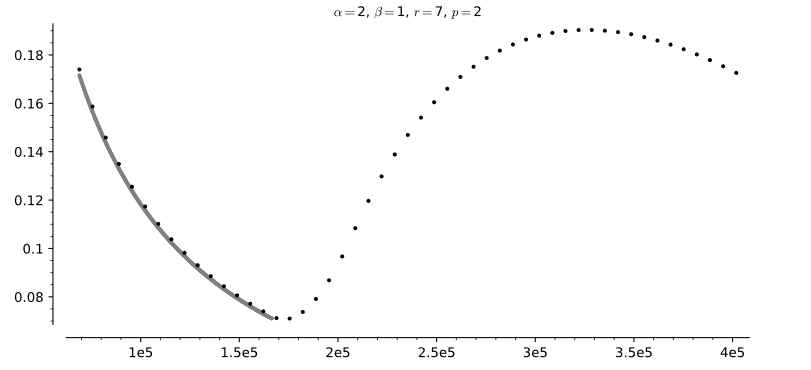}

\Graph{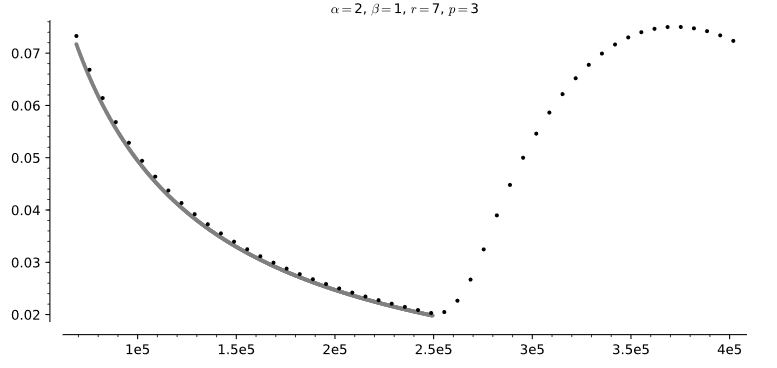}

\Graph{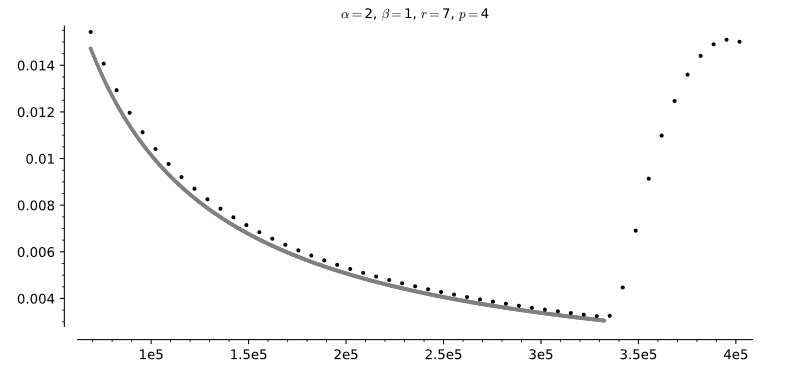}

We now show all four of these data plots for $\al=2,\be=1$ together in a single plot.

\Graph{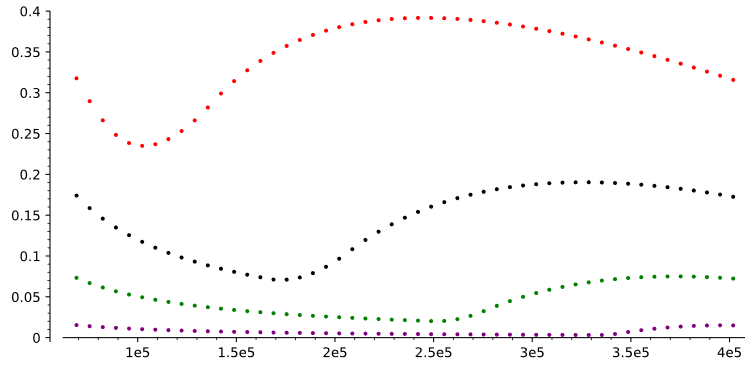}

We next show data for $\al=1,\ \be=5$, $r=4$, and $5\le p\le 6$.  We note that amongst the data we show, the $p=5$ case here is the only case where the simpler formula of Corollary~\ref{C-pLarge} does not apply.

\Graph{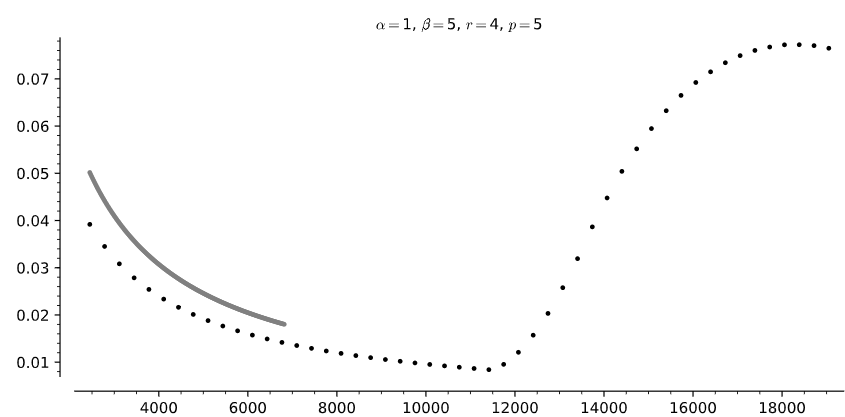}

\Graph{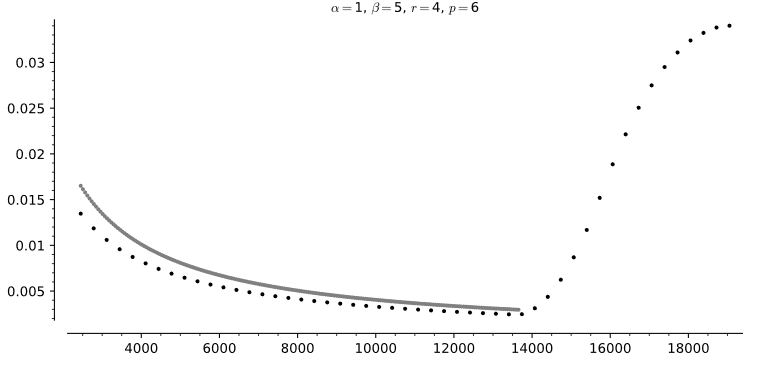}

Finally, we show all of the data plots for $1\le p\le 6$.

\Graph{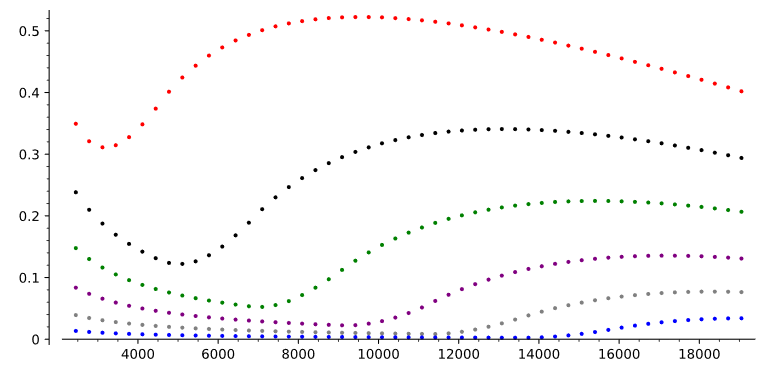}

We next present data regarding slowest slow walks.  Let $i_{\al,\be}(n)=|\{m:m\le n,\ (\al,\be)\in \S(m)\}|$.  Note that Proposition~\ref{P-Slowest} gives an upper bound on this values.  Below we give a plot of $i_{1,2}(n)$ for $1\le n\le 50,000$.  We note that $i_{1,3}(n),\ i_{1,4}(n)$, and $i_{2,1}(n)$ are all too small in this range to extrapolate anything from their plots.

\Graph{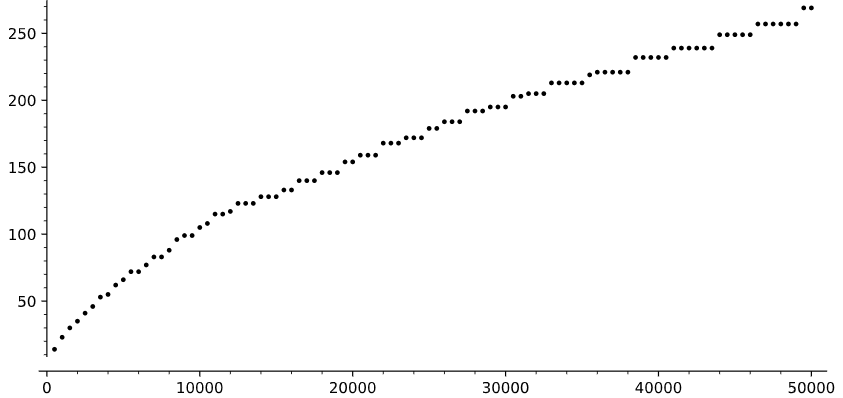}

Below is a log-log plot of $i_{1,2}(n)$ in the range $30,000\le n\le 50,000$.  The line is $y=(2-\log_{\phi}(2))x-.47$.  As $\gam_{1,2}=2$, this may suggest that in this case the upper bound of Proposition~\ref{P-Slowest} is close to tight. 

\Graph{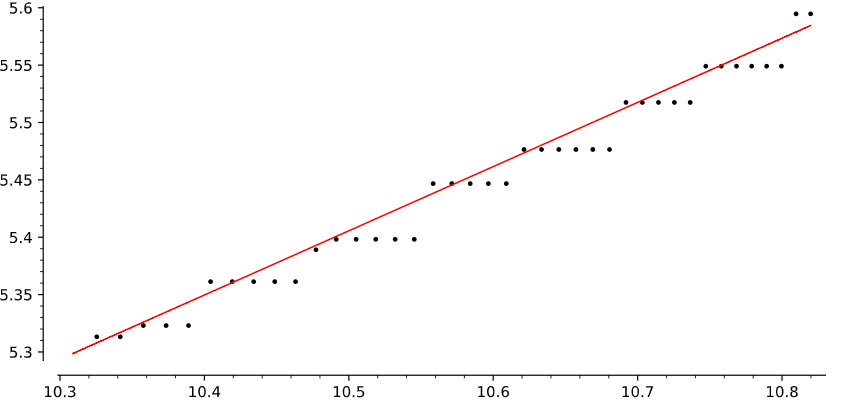}

Let $T=\{(1,6),(2,3)\}$ and $e_{\al,\be}(n)=n^{-1}|\{m:m\le n,\S_T(m)=\{(\al,\be)\}\}|$.  Recall that for this $T$ Corollary~\ref{C-Slowest} does not apply.  Below we plot $e_{1,6}(n)$ with solid blue dots and $e_{2,3}(n)$ with open black circles for $1\le n\le 500,000$.

\Graph{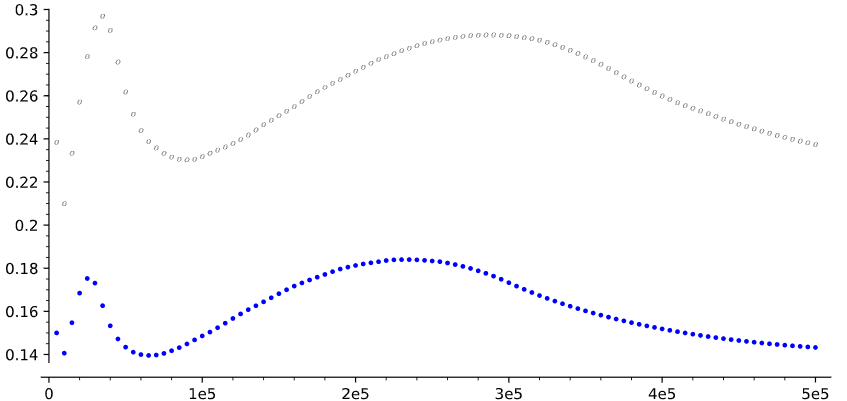}

Based on this data, it seems that both of these values are $\Om(1)$, that $e_{2,3}(n)\ge e_{1,6}(n)$ for all $n$, and that these exact values oscillate in some way.  We emphasize that we do not have proofs of any of these observations.  We end this section with a plot of the ratio $e_{2,3}(n)/e_{1,6}(n)$.

\Graph{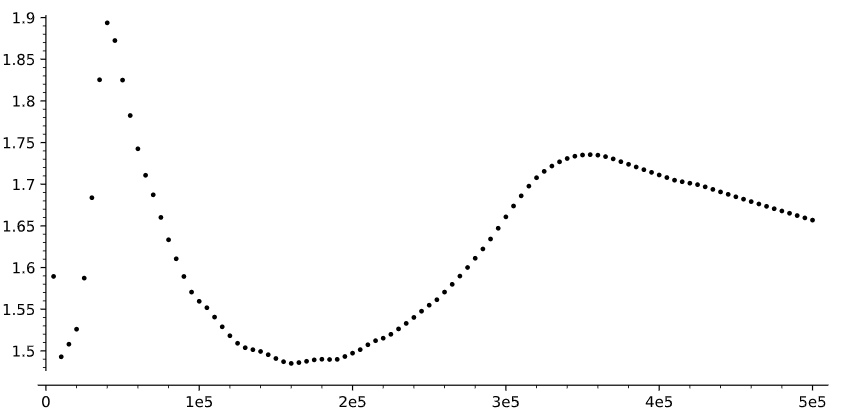}

\end{document}